\documentclass[leqno,12pt]{amsart} 
\usepackage{amssymb}
\theoremstyle{plain} 
\newtheorem{theorem}{\indent\sc Theorem}[section]
\newtheorem{lemma}[theorem]{\indent\sc Lemma}

\theoremstyle{definition} 

\newtheorem{remark}[theorem]{\indent\sc Remark}

\begin{document}

\font\eightrm=cmr8
\font\eightit=cmti8
\font\eighttt=cmtt8
\def\nft
{\hbox{$n$\hskip3pt$\equiv$\hskip4pt$5$\hskip4.4pt$($mod\hskip2pt$3)$}}
\def\bbP{\text{\bf P}}
\def\bbR{\text{\bf R}}
\def\rto{\mathbf{R}\hskip-.5pt^2}
\def\rtr{\text{\bf R}\hskip-.7pt^3}
\def\rn{{\mathbf{R}}^{\hskip-.6ptn}}
\def\rk{{\mathbf{R}}^{\hskip-.6ptk}}
\def\bbZ{\text{\bf Z}}
\def\hyp{\hskip.5pt\vbox
{\hbox{\vrule width3ptheight0.5ptdepth0pt}\vskip2.2pt}\hskip.5pt}
\def\er{r\hskip.3pt}
\def\df{d\hskip-.8ptf}
\def\dfh{d\hskip-.8pt\fh}
\def\txm{{T\hskip-2pt_x\hskip-.9ptM}}
\def\txhm{{T\hskip-2pt_x\hskip-.9pt\hm}}
\def\txn{{T\hskip-2pt_x\hskip-.9ptN}}
\def\txthm{{T\hskip-2pt_{x(t)}\hskip-.9pt\hm}}
\def\txohm{{T\hskip-2pt_{x(0)}\hskip-.9pt\hm}}
\def\txtsm{{T\hskip-2pt_{x(t,s)}\hskip-.9ptM}}
\def\tm{{T\hskip-.3ptM}}
\def\tn{{T\hskip-.3ptN}}
\def\tam{{T^*\!M}}
\def\so{\mathfrak{so}\hh}
\def\gi{\mathfrak{g}} 
\def\hi{\mathfrak{h}}
\def\fh{f}
\def\hf{\varLambda}
\def\rc{c}
\def\rd{d}
\def\kri{\text{\rm Ker}\hskip1pt\ri}
\def\kr{\text{\rm Ker}\hskip1ptR}
\def\kw{\text{\rm Ker}\hskip2ptW}
\def\kb{\text{\rm Ker}\hskip1.5ptB}
\def\xc{\mathcal{X}_c}
\def\dz{\mathcal{D}}
\def\dzp{\dz^\perp}
\def\dzx{\dz\hskip-.3pt_x}
\def\dzxp{\dz\hskip-.3pt_x{}\hskip-4.2pt^\perp}
\def\dzxa{\dz\hskip-.3pt_x{}\hskip-4.2pt^*}
\def\lz{\mathcal{L}}
\def\tlz{\tilde{\lz}}
\def\xe{\mathcal{E}}
\def\xz{\mathcal{X}}
\def\yz{\mathcal{Y}}
\def\zz{\mathcal{Z}}
\def\tim{\hskip1.5pt\widetilde{\hskip-1.5ptM\hskip-.5pt}\hskip.5pt}
\def\tig{\hskip.7pt\widetilde{\hskip-.7ptg\hskip-.4pt}\hskip.4pt}
\def\hm{\hskip1.9pt\widehat{\hskip-1.9ptM\hskip-.2pt}\hskip.2pt}
\def\hmt{\hskip1.9pt\widehat{\hskip-1.9ptM\hskip-.5pt}_t}
\def\hmz{\hskip1.9pt\widehat{\hskip-1.9ptM\hskip-.5pt}_0}
\def\hmp{\hskip1.9pt\widehat{\hskip-1.9ptM\hskip-.5pt}_p}
\def\hg{\hskip1.2pt\widehat{\hskip-1.2ptg\hskip-.4pt}\hskip.4pt}
\def\hri{\hskip.7pt\widehat{\hskip-.7pt\ri\hskip-.4pt}\hskip.4pt}
\def\hn{\hskip.7pt\widehat{\hskip-.7pt\nabla\hskip-.4pt}\hskip.4pt}
\def\nao{\hbox{$\nabla\!\!^{^{^{_{\!\!\circ}}}}$}}
\def\ro{\hbox{$R\hskip-4.5pt^{^{^{_{\circ}}}}$}{}}
\def\mppp{\hbox{$-$\hskip1pt$+$\hskip1pt$+$\hskip1pt$+$}}
\def\mpdp{\hbox{$-$\hskip1pt$+$\hskip1pt$\dots$\hskip1pt$+$}}
\def\mmpp{\hbox{$-$\hskip1pt$-$\hskip1pt$+$\hskip1pt$+$}}
\def\mmmp{\hbox{$-$\hskip1pt$-$\hskip1pt$-$\hskip1pt$+$}}
\def\pppp{\hbox{$+$\hskip1pt$+$\hskip1pt$+$\hskip1pt$+$}}
\def\mpmp{\hbox{$-$\hskip1pt$\pm$\hskip1pt$+$}}
\def\mpmpp{\hbox{$-$\hskip1pt$\pm$\hskip1pt$+$\hskip1pt$+$}}
\def\mmpmp{\hbox{$-$\hskip1pt$-$\hskip1pt$\pm$\hskip1pt$+$}}
\def\q{q}
\def\bq{\hat q}
\def\p{p}
\def\w{\vt^\perp}
\def\x{v}
\def\y{y}
\def\vp{\vt^\perp}
\def\vd{\vt\hh'}
\def\vdx{\vd{}\hskip-4.5pt_x}
\def\bz{b\hh}
\def\cy{{y}}
\def\rwo{\,\hs\text{\rm rank}\hskip2.7ptW\hskip-2.7pt=\hskip-1.2pt1}
\def\rwho{\,\hs\text{\rm rank}\hskip2.2ptW^h\hskip-2.2pt=\hskip-1pt1}
\def\rw{\,\hs\text{\rm rank}\hskip2.4ptW\hskip-1.5pt}
\def\im{\varPhi}
\def\js{J}
\def\ism{H}
\def\fe{F}
\def\fy{f}
\def\dfc{dF\hskip-2.3pt_\cy\hskip.4pt}
\def\dfct{dF\hskip-2.3pt_\cy(t)\hskip.4pt}
\def\dic{d\im\hskip-1.4pt_\cy\hskip.4pt}
\def\vl{\Lambda}
\def\qt{\mathcal{E}}
\def\tqt{\tilde{\qt}}
\def\vh{h}
\def\mv{V}
\def\vy{\mathcal{V}}
\def\xv{\mathcal{X}}
\def\yv{\mathcal{Y}}
\def\iv{\mathcal{I}}
\def\gkp{\Sigma}
\def\bs{\varSigma}
\def\hs{\hskip.7pt}
\def\hh{\hskip.4pt}
\def\nh{\hskip-.7pt}
\def\nnh{\hskip-1pt}
\def\hrz{^{\hskip.5pt\text{\rm hrz}}}
\def\vrt{^{\hskip.2pt\text{\rm vrt}}}
\def\th{\varTheta}
\def\zh{\zeta}
\def\vg{\varGamma}
\def\my{\mu}
\def\ny{\nu}
\def\gy{\lambda}
\def\gm{\gamma}
\def\gp{\mathrm{G}}
\def\hp{\mathrm{H}}
\def\kp{\mathrm{K}}
\def\Gm{\Gamma}
\def\Lm{\Lambda}
\def\Dt{\Delta}
\def\sj{\sigma}
\def\lg{\langle}
\def\rg{\rangle}
\def\lr{\lg\,,\rg}
\def\uv{\underline{v\hskip-.8pt}\hskip.8pt}
\def\uvp{\underline{v\hh'\hskip-.8pt}\hskip.8pt}
\def\uw{\underline{w\hskip-.8pt}\hskip.8pt}
\def\uxs{\underline{x_s\hskip-.8pt}\hskip.8pt}
\def\vs{vector space}
\def\rvs{real vector space}
\def\vf{vector field}
\def\tf{tensor field}
\def\tvn{the vertical distribution}
\def\dn{distribution}
\def\od{Ol\-szak distribution}
\def\pt{point}
\def\tc{tor\-sion\-free connection}
\def\ea{equi\-af\-fine}
\def\rt{Ricci tensor}
\def\pde{partial differential equation}
\def\pf{projectively flat}
\def\pfs{projectively flat surface}
\def\pfc{projectively flat connection}
\def\pftc{projectively flat tor\-sion\-free connection}
\def\su{surface}
\def\sco{simply connected}
\def\psr{pseu\-\hbox{do\hs-}Riem\-ann\-i\-an}
\def\inv{-in\-var\-i\-ant}
\def\trinv{trans\-la\-tion\inv}
\def\feo{dif\-feo\-mor\-phism}
\def\feic{dif\-feo\-mor\-phic}
\def\feicly{dif\-feo\-mor\-phi\-cal\-ly}
\def\Feicly{Dif\-feo\-mor\-phi\-cal\-ly}
\def\diml{-di\-men\-sion\-al}
\def\prl{-par\-al\-lel}
\def\skc{skew-sym\-met\-ric}
\def\sky{skew-sym\-me\-try}
\def\Sky{Skew-sym\-me\-try}
\def\dbly{-dif\-fer\-en\-ti\-a\-bly}
\def\cs{con\-for\-mal\-ly symmetric}
\def\cf{con\-for\-mal\-ly flat}
\def\ls{locally symmetric}
\def\ecs{essentially con\-for\-mal\-ly symmetric}
\def\rr{Ric\-ci-re\-cur\-rent}
\def\kf{Killing field}
\def\om{\omega}
\def\vol{\varOmega}
\def\dv{\delta}
\def\ve{\varepsilon}
\def\zt{\zeta}
\def\kx{\kappa}
\def\mf{manifold}
\def\mfd{-man\-i\-fold}
\def\bmf{base manifold}
\def\bd{bundle}
\def\tbd{tangent bundle}
\def\ctb{cotangent bundle}
\def\bp{bundle projection}
\def\prc{pseu\-\hbox{do\hs-}Riem\-ann\-i\-an metric}
\def\prd{pseu\-\hbox{do\hs-}Riem\-ann\-i\-an manifold}
\def\Prd{pseu\-\hbox{do\hs-}Riem\-ann\-i\-an manifold}
\def\npd{null parallel distribution}
\def\pj{-pro\-ject\-a\-ble}
\def\pd{-pro\-ject\-ed}
\def\lcc{Le\-vi-Ci\-vi\-ta connection}
\def\vb{vector bundle}
\def\vbm{vec\-tor-bun\-dle morphism}
\def\kerd{\text{\rm Ker}\hskip2.7ptd}
\def\sy{\sigma}
\def\ts{total space}
\def\pmb{\pi}
\def\ri{{\rho}}

\renewcommand{\thepart}{\Roman{part}}

\title[Compact manifolds with parallel Weyl tensor]{On compact manifolds 
admitting indefinite\\
metrics with parallel Weyl tensor}
\author[A. Derdzinski]{Andrzej Derdzinski} 

\author[W. Roter]{Witold Roter} 

%
\address{
Department of Mathematics \endgraf
The Ohio State University \endgraf
Columbus, OH 43210 \endgraf
USA
}
\email{andrzej@math.ohio-state.edu}

\address{
Institute of Mathematics and Computer Science \endgraf
Wroc\l aw University of Technology \endgraf
Wy\-brze\-\.ze Wys\-pia\'n\-skiego 27, 50-370 Wroc\l aw \endgraf
Poland
}
\email{roter@im.pwr.wroc.pl}

\begin{abstract}Compact pseudo-Riemannian manifolds that have parallel Weyl 
tensor without being conformally flat or locally symmetric are known to exist 
in infinitely many dimensions greater than 4. We prove some general 
topological properties of such manifolds, namely, vanishing of the Euler 
characteristic and real Pontryagin classes, and infiniteness of the 
fundamental group. We also show that, in the Lorentzian case, each of them 
is at least 5-dimensional and admits a two-fold cover which is a bundle over 
the circle.
\end{abstract}

\maketitle

\voffset=-33pt\hoffset=-6pt

\setcounter{theorem}{0}
\renewcommand{\thetheorem}{\Alph{theorem}}
\section*{Introduction}
One calls a \prd\ $\,(M,g)\,$ of dimension $\,n\ge4\,$ {\it conformally 
symmetric\/} \cite{chaki-gupta} if its Weyl conformal tensor is parallel, and 
{\it essentially conformally symmetric\/} if, in addition, $\,(M,g)\,$ is 
neither conformally flat nor locally symmetric. All \ecs\ \mf s have 
indefinite metrics \cite[Theorem~2]{derdzinski-roter-77}.

The Weyl conformal tensor is one of the three irreducible components of the 
curvature tensor under the action of the pseu\-do\hs-or\-thog\-o\-nal group, 
the other two corresponding to the scalar curvature and traceless Ricci 
tensor. This puts \cs\ \mf s on par with two other classes, formed by \mf s 
with constant scalar curvature and, respectively, parallel Ric\-ci tensor, 
including Einstein spaces.

The local structure of \ecs\ \prc s is fully understood 
\cite{derdzinski-roter-07}; they realize any prescribed indefinite signature 
in every dimension $\,n\ge4$. They are also known to exist on some compact 
manifolds \feic\ to torus bundles over the circle \cite{derdzinski-roter}, 
where they represent all indefinite metric signatures in all dimensions 
$\,n\ge5\,$ such that $\,\nft$. Consequently, there arises a natural question 
of characterizing the compact manifolds that admit such metrics.

The present paper provides a step toward an answer by establishing some 
necessary conditions. Our first result, except for the claim about $\,\pi_1M$, 
is derived in Section~\ref{pfta} from the Chern\hs-Weil formulae. 
Infiniteness of $\,\pi_1M\,$ is proved in Section~\ref{psta}: we argue there 
that, if $\,M\,$ were simply connected, it would be a bundle over $\,S^2$ with 
a fibre covered by $\,\bbR^{\hskip-.6ptn-2}\nnh$, which is impossible for 
topological reasons.
\begin{theorem}\label{pontr}Let a manifold\/ $\hs M\nnh$ of 
\hskip-.3ptdimension $\hskip1ptn\ge4$ admit an \ecs\ \prc. The real Pontryagin 
classes $\,p_i(M)\in H^{4i}(M,\bbR)$ then vanish for all\/ $\,i\ge1$, If, in 
addition, $\,M\hs$ is compact, then it has zero Euler characteristic, and its 
fundamental group is infinite.
\end{theorem}
Applied to manifolds $\,M\,$ which are both compact and orientable, the 
assertion about $\,p_i(M)\,$ in Theorem~\ref{pontr} leads to a similar 
conclusion about the Pontryagin numbers, including the signature of $\,M$.

As mentioned above, compact \ecs\ Lo\-rentz\-i\-an \mf s exist in infinitely 
many dimensions $\,n$, starting from $\,n=5$, while noncompact ones exist in 
every dimension $\,n\ge4$. Our next two results deal with their topological 
structure in the compact case and with the dimension $\,4$.
\begin{theorem}\label{bdcir}Let\/ $\,(M,g)\,$ be a compact \ecs\ 
Lo\-rentz\-i\-an \mf. Then some two-fold covering manifold of\/ $\,M\,$ is 
the total space of a $\,C^\infty\nnh$ bundle over the circle, the fibre of 
which admits a flat tor\-sion\-free connection with a nonzero parallel vector 
field.
\end{theorem}
\begin{theorem}\label{nolor}Every four\diml\ \ecs\ Lo\-rentz\-i\-an \mf\ is 
noncompact.
\end{theorem}
The fibration $\,M\to S^1$ in Theorem~\ref{bdcir} has an explicit geometric 
description. Namely, its fibres are the leaves of a parallel distribution 
$\,\dzp$ on $\,(M,g)$, which is the orthogonal complement of the {\it Ol\-szak 
distribution} $\,\dz$, defined, in any \cs\ \mf, by declaring the sections of 
$\,\dz\hs$ to be the vector fields $\,u\,$ such that, for all vector fields 
$\,v,v\hh'\nnh$, one has $\,\xi\wedge\hs\varOmega\hs=0$, where 
$\,\xi=g(u,\,\cdot\,)\,$ and 
$\,\hs\varOmega\hs=W(v,v\hh'\nnh,\,\cdot\,,\,\cdot\,)$. (Here $\,W\hs$ denotes 
the Weyl tensor; thus, $\,\hs\varOmega\,$ is a differential $\,2$-form.) 
Ol\-szak introduced the distribution $\,\dz\,$ in a more general situation 
\cite{olszak}, and showed that, on an \ecs\ \mf, $\,\dz\,$ is a \npd\ of 
dimension $\,1\,$ or $\,2$. See Lemma~\ref{csnpd}(i).

We derive Theorems~\ref{bdcir} and~\ref{nolor} in Sections~\ref{potb} 
and~\ref{potc} from the following result, proved in Section~\ref{pote}:
\begin{theorem}\label{parvf}For every \ecs\ Lo\-rentz\-i\-an \mf, the \od\/ 
$\,\dz$ is one\diml. Passing to a two-fold covering manifold, if 
necessary, we may assume that\/ $\,\dz\hs$ is trivial as a real line bundle, 
and then\/ $\,\dz\hs$ is spanned by a global parallel vector field.
\end{theorem}

\renewcommand{\thetheorem}{\thesection.\arabic{theorem}}
\section{Preliminaries}\label{prel}
Throughout this paper, all manifolds and bundles, along with sections and 
connections, are assumed to be of class $\,C^\infty\nnh$. Manifolds 
(including fibres of bundles) are, by definition, connected. A mapping is 
always a $\,C^\infty$ mapping betweeen manifolds.
\begin{remark}\label{fibra}A surjective submersion $\,\pi:M\to P\,$ such 
that the sets $\,\pi^{-1}(y)$, $\,y\in P$, are all compact can always be 
factored as $\,M\to Q\to P$, with a locally trivial fibration $\,M\to Q\,$ 
having compact (connected) fibres, and a finite covering projection 
$\,Q\to P$.

Namely, $\,\pi\,$ itself is a locally trivial fibration, except that each 
fibre $\,\pi^{-1}(y)$, rather than being connected (and hence a manifold in 
our sense), may in general have some finite set $\,Q_y$ of connected 
components. This well-known fact becomes clear if one uses the holonomy of 
any ``nonlinear connection'' (a distribution in $\,M\,$ complementary to the 
fibres). We now define $\,Q\,$ and $\,Q\to P\,$ by 
$\,Q=\bigcup_{\hh y\in P}(\{y\}\times Q_y)\,$ and $\,(y,N)\mapsto y$.
\end{remark}
\begin{lemma}\label{clext}Suppose that a closed\/ $\,1$-form\/ $\,\xi$ on 
a compact manifold\/ $\,M\hs$ is nonzero everywhere and, for some function 
$\,\phi:M\to\bbR\,$ which is nonzero somewhere, the form\/ $\,\phi\hs\xi\,$ 
is exact. Then\/ $\,M\hs$ is the total space of a bundle over the circle 
$\,S^1\nnh$. In addition, for any functions\/ $\,t:\hm\to\bbR\,$ on the 
universal covering \mf\ of\/ $\,M\hs$ and\/ $\,\theta:M\to\bbR\hs$, such 
that\/ $\,dt\,$ is the pull\-back of\/ $\,\xi\,$ to $\,\hm\,$ and\/ 
$\,d\theta=\phi\hs\xi$, and for some $\,c\in(0,\infty)$,
\begin{enumerate}
  \def\theenumi{{\rm\alph{enumi}}}
\item $t:\hm\to\bbR\,$ descends to a bundle projection 
$\,M\to\bbR/c\hh\bbZ\hh=S^1\nnh$, 
\item the pull\-back of\/ $\,\theta\,$ to\/ $\,\hm\hs$ equals the composite 
$\,\hf(t)\,$ for some nonconstant function $\,\hf:\bbR\to\bbR\,$ which is 
periodic, and has\/ $\,c\,$ as a period,
\item $\hm\,$ can be \feicly\ identified with $\,\bbR\times N\hs$ for some 
manifold\/ $\,N\hs$ so as to make\/ $\,\hs t\,$ coincide with the 
projection\/ $\,\bbR\times N\to\bbR\hs$.
\end{enumerate}
\end{lemma}
\begin{proof}Let $\,\hg\,$ be the pull\-back to $\,\hm\,$ of any fixed 
Riemannian metric $\,g\,$ on $\,M$. The $\,\hg$-gra\-di\-ent $\,\hn t\,$ of 
$\,t\,$ gives rise to the vector field $\,w=\hn t/\hg(\hn t,\hn t)$, which is 
complete, being the pull\-back of the vector field $\,u/g(u,u)\,$ on $\,M$, 
for $\,u\,$ such that $\,\xi=g(u,\,\cdot\,)$. A standard argument 
\cite[p.\ 12]{milnor} using the flow of $\,w$, for which $\,t\,$ itself serves 
as the parameter, yields (c). Denoting by $\,\widehat\theta\,$ and 
$\,\widehat\phi\,$ the pull\-backs of $\,\theta\hs$ and $\,\phi\hs$ to $\,\hm$, 
we have $\,\hn\widehat\theta=\widehat\phi\hs\hn t$. Hence $\,\widehat\theta\,$ 
is, locally, a function of $\,t$. The word `locally' can 
in turn be dropped as the level sets of $\,t\,$ are connected by (c). Thus, 
$\,\widehat\theta=\hf(t)\,$ (that is, $\,\widehat\theta=\hf\nnh\circ t$) for 
some function $\,\hf:\bbR\to\bbR\hs$. Since $\,\,\theta\,$ and 
$\,\,\widehat\theta\,$ are nonconstant, so is $\,\hf$.

The invariance of $\,dt\,$ under the action of the deck transformation group 
$\,\Gm=\pi_1M$ implies that $\,t\circ\alpha=t+\varXi(\alpha)\,$ for some 
homomorphism $\,\varXi:\Gm\to\bbR\,$ and all $\,\alpha\in\Gm$. As 
$\,\widehat\theta=\hf(t)\,$ is $\,\Gm$-in\-var\-i\-ant, every nonzero value of 
$\,\varXi\,$ is a period of $\,\hf$. Thus, $\,\varXi(\Gm)=c\hh\bbZ$ for 
some $\,c\in(0,\infty)$, and (b) follows. Finally, the surjective submersion 
$\,t:\hm\to\bbR$ descends to a mapping $\,M=\hm\nnh/\Gm\to\bbR/c\hh\bbZ$, 
which must be a surjective submersion as well. In addition, for each 
$\,s\in\bbR\hs$, the pre\-im\-age of $\,s+c\hh\bbZ\,$ under the latter mapping 
is connected, as it coincides with the image of $\,t^{-1}(s)\subset\hm\,$ 
under the covering projection $\,\hm\to M$, and $\,t^{-1}(s)\subset\hm\,$ 
is connected by (c). Combined with Remark~\ref{fibra}, this proves (a).
\end{proof}
Given a connection $\,\nabla\,$ in a \vb\ $\,\xe\hs$ over a \mf\ $\,M$, 
a section $\,\psi\,$ of $\,\xe$, and \vf s $\,u,v\,$ tangent to $\,M$, 
our sign convention for the curvature tensor $\,R=R^\nabla$ is 
\begin{equation}\label{cur}
R(u,v)\psi\hskip7pt=\hskip7pt\nabla_{\!v}\nabla_{\!u}\psi\,
-\,\nabla_{\!u}\nabla_{\!v}\psi\,+\,\nabla_{[u,v]}\psi\hs.
\end{equation}
Such $\,\nabla\nh,\hs\xe,M,u\,$ and $\,v\,$ give rise to the bundle morphism
\begin{equation}\label{cvo}
R^\nabla\nnh(u,v):\xe\,\to\,\xe
\end{equation}
sending a section $\,\psi\,$ of $\,\xe\hs$ to 
$\,R^\nabla\nnh(u,v)\psi=R(u,v)\psi\,$ defined 
in (\ref{cur}).

We always denote by $\,\nabla\,$ both the \lcc\ of a given \prd\ $\,(M,g)$, 
and the $\,g$-gra\-di\-ent operator. The same symbol $\,\nabla\,$ is also 
used for connections induced by $\,\nabla\,$ in $\,\nabla$\prl\ subbundles of 
$\,\tm\,$ and their quotients.

A {\it pseu\-\hbox{do\hs-}Riem\-ann\-i\-an fibre metric\/} $\,\gm\,$ in a \vb\ 
$\,\mathcal{E}\,$ over a \mf\ $\,M\,$ is, as usual, any family of 
nondegenerate symmetric bilinear forms $\,\gm_x$ in the fibres 
$\,\mathcal{E}_x$ that constitutes a $\,C^\infty$ section of the symmetric 
power $\,(\mathcal{E}^*){}^{\odot2}\nnh$.
\begin{remark}\label{scfco}Every simply connected manifold $\,N$ with a 
complete flat tor\-sion\-free connection $\,\nabla\,$ is \feic\ to a 
Euclidean space: the exponential mapping of $\,\nabla\hs$ at any point 
$\,x\in N\hs$ is an affine \feo\ $\,\txn\to N\nh$. See 
\cite[p.\ 145]{auslander-markus}.
\end{remark}
Given a flat connection $\,\nabla\,$ in a \vb\ $\,\xe\hs$ of fibre dimension 
$\,k\,$ over a \mf\ $\,M$, we will say that $\,\xe\hs$ {\it is trivialized by 
its parallel sections\/} if the space of $\,\nabla$\prl\ global sections of 
$\,\xe\hs$ is $\,k$\diml\ (or, in other words, $\,\nabla\,$ is {\it globally} 
flat).
\begin{lemma}\label{flcpl}If a compact manifold\/ $\,N$ with a flat 
tor\-sion\-free connection\/ $\,\nabla\hs$ admits a $\,\nabla$\prl\ 
distribution $\,\lz\hs$ such that both bundles $\,\lz\hs$ and\/ 
$\,\qt=\tn\nh/\lz$, with the flat connections induced by $\,\nabla$, are 
trivialized by their parallel sections, then\/ $\,\nabla\hs$ is complete.
\end{lemma}
\begin{proof}We denote by $\,\xv\,$ the space of parallel sections of 
$\,\lz$, and by $\,\hs\exp_{\hs x}$ the exponential mapping of $\,\nabla\hs$ 
at $\,x\in M$. Geodesics tangent to $\,\lz\,$ at some (or every) point, being 
integral curves of elements of $\,\xv$, are obviously complete.
 
Let $\,\yv\,$ be a vector subbundle of $\,\tn\hs$ such that 
$\,\tn=\lz\oplus\yv$, and let us fix $\,w\in\mv\nh$, where $\,\mv\hs$ is 
the vector space of sections of $\,\yv\,$ obtained as the image of the 
space of parallel sections of $\,\qt\hs$ under the obvious isomorphism 
$\,\qt\to\yv$. The vector field $\,v=\nabla_{\!w}w\,$ then is a section of 
$\,\lz$. (In fact, locally, $\,w=w\hh'\nh+\widetilde{\nh w\nh}$, for a 
section $\,w\hh'$ of $\,\lz\,$ and a local parallel \vf\ 
$\,\widetilde{\nh w\nh}$, so that $\,\nabla_{\!w}w=\nabla_{\!w}w\hh'\nnh$, 
while $\,\nabla_{\!w}w\hh'$ is a section of $\,\lz$.) Any integral curve 
$\,\bbR\ni t\mapsto x(t)\in N\hs$ of $\,w\,$ now gives rise to a function 
$\,\zeta:\bbR\to\xv\,$ defined by requiring that $\,\zeta(t)\in\xv\,$ have the 
value $\,v_{x(t)}$ at the point $\,x(t)$. Any function $\,\eta:\bbR\to\xv$ 
with the second derivative $\,\ddot\eta=-\hs\zeta\,$ leads in turn to the 
curve $\,\bbR\ni t\mapsto y(t)\in N\nh$, given by 
$\,y(t)=\exp_{\hs x(t)}\hh\eta(t,x(t))$, where $\,\eta(t,x)\in \txm\,$ is 
the value of $\,\eta(t)\in\xv\,$ at $\,x\in M$. That $\,t\mapsto y(t)\,$ is a 
geodesic is clear: a treating $\,N\nh$, locally, as an affine space, we have 
$\,y(t)=x(t)+\eta(t)$, and $\,\ddot y=\ddot x+\ddot\eta=0$. Since such 
geodesics realize all initial data, our assertion follows.
\end{proof}

Let $\,(t,s)\mapsto x(s,t)\,$ be a fixed {\it variation of curves} in a \prd\ 
$\,(M,g)$, that is, an $\,M$-val\-ued $\,C^\infty$ mapping from a rectangle 
(product of intervals) in the $\,ts\hh$-plane. By a {\it vector field\/ $\,w\,$ 
along the variation} we mean, as usual, a section of the pull\-back of 
$\,\tm\,$ to the rectangle (so that $\,w(t,s)\in\txtsm$). Examples are $\,x_s$ 
and $\,x_t$, which assign to $\,(t,s)\,$ the velocity of the curve 
$\,t\mapsto x(t,s)\,$ (or, $\,s\mapsto x(t,s)$) at $\,s\,$ (or $\,t$). Further 
examples are provided by restrictions to the variation of vector fields on 
$\,M$. The partial covariant derivatives of a vector field $\,w\,$ along the 
variation are the vector fields $\,w_t,\hs w_s$ along the variation, obtained 
by differentiating $\,w\,$ covariantly along the curves $\,t\mapsto x(t,s)\,$ 
or $\,s\mapsto x(t,s)$. Skipping parentheses, we write $\,w_{ts},\hs w_{stt}$, 
etc., rather than $\,(w_t){}_s,\hs ((w_s){}_t){}_t$ for higher-or\-der 
derivatives, as well as $\,x_{ss},\hs x_{st}$ instead of 
$\,(x_s){}_s,\hs (x_s){}_t$. One always has $\,w_{ts}=w_{st}+R(x_t,x_s)\hh w$, 
cf.\ \cite[formula (11.2) on p.\ 493]{dillen-verstraelen}, and, since the 
\lcc\ $\,\nabla\,$ is tor\-sion\-free, $\,x_{st}=x_{ts}$. Consequently, 
$\,x_{ttss}\,=\,x_{stts}\,+\,[R(x_t,x_s)x_t]_s$, 
$\,x_{stts}\,=\,x_{stst}\,+\,R(x_t,x_s)x_{st}$ and 
$\,x_{stst}\,=\,x_{sstt}\,+\,[R(x_t,x_s)x_s]_t$. Thus, whenever 
$\,(t,s)\mapsto x(s,t)\,$ is a variation of curves in $\,M$,
\begin{equation}\label{xts}
\begin{array}{rl}
\mathrm{a)}&\hskip10ptx_{tts}\,\,=\,\,x_{stt}\,+\,\hs R(x_t,x_s)x_t\hs,\\
\mathrm{b)}&\hskip10ptx_{ttss}\,\,=\,\,x_{sstt}\,+\,\hs[R(x_t,x_s)x_s]_t
+\,\hs R(x_t,x_s)x_{st}\,+\,\hs[R(x_t,x_s)x_t]_s\hs.
\end{array}
\end{equation}

\section{Con\-for\-mal\-ly symmetric manifolds}\label{csma}
The Schouten tensor $\,\sigma\,$ and Weyl conformal tensor $\,W\hs$ of a \prd\ 
$\,(M,g)\,$ of dimension $\,n\ge4\,$ are given by 
$\,\sigma=\ri\,-\hs(2n-2)^{-1}\,\text{\rm s}\hskip1.2ptg$, with $\,\ri\,$ 
denoting the Ric\-ci tensor, 
$\,\hs\text{\rm s}\hs=\hs\text{\rm tr}_g\hh\ri\hs\,$ standing for the 
scalar curvature, and $\,W\nh=R-(n-2)^{-1}\hs g\wedge\hh\sigma$. Here 
$\,\wedge\,$ is the exterior multiplication of $\,1$-forms valued in 
$\,1$-forms, which involves the ordinary $\,\wedge\,$ as the val\-ue\-wise 
multiplication; thus, $\,g\wedge\hh\sigma\,$ is a $\,2$-form valued in 
$\,2$-forms.
\begin{lemma}\label{rwcod}For any \ecs\ manifold of dimension $\,n\ge4$,
\begin{enumerate}
  \def\theenumi{{\rm\alph{enumi}}}
\item $R\hs\,=\,\hs W\hs+\,(n-2)^{-1}\hs g\wedge\ri$,
\item the Ric\-ci tensor $\,\ri\,$ satisfies the Co\-daz\-zi equation, in the 
sense that the three-times covariant tensor field\/ $\,\nabla\nh\ri\,$ is 
totally symmetric.
\end{enumerate}
\end{lemma}
\begin{proof}In any \ecs\ manifold, $\,\hs\text{\rm s}\hs=0\,$ identically 
\cite[Theorem~7, p.\ 21]{derdzinski-roter-78}, so that $\,\sigma=\ri$. This 
gives (a), and (b) follows since the condition $\,\nabla\hs W\nnh=0$ implies 
vanishing of the divergence of $\,W\nh$, which, in view of the second Bianchi 
identity, is equivalent to the Co\-daz\-zi equation for $\,\sigma$, cf.\ 
\cite[formula (5.29) on p.\ 460]{dillen-verstraelen}.
\end{proof}
Assertion (i) in the next lemma is due to Ol\-szak \cite{olszak}.
\begin{lemma}\label{csnpd}Let\/ $\,\dz\,$ be the \od\ of an \ecs\ \mf\/ 
$\,(M,g)$, defined in the Introduction. Then 
\begin{enumerate}
  \def\theenumi{{\rm\roman{enumi}}}
\item $\dz\,$ is a \npd\ of dimension\/ $\,1\,$ or\/ $\,2$,
\item at every point\/ $\,x$, the space $\,\dz_x$ contains the image of the 
Ric\-ci tensor treated, with the aid of\/ $\,g_x$, as a linear operator 
$\,\txm\to \txm\nh$,
\item $R(v,v\hh'\nnh,\,\cdot\,,\,\cdot\,)
=\hs W(v,v\hh'\nnh,\,\cdot\,,\,\cdot\,)\hs\,=\,\hs0\hs\,$ whenever $\,v\,$ and 
$\,v\hh'$ are sections of $\,\dzp\nnh$,
\item of the connections in the vector bundles $\,\dz\hs$ and\/ 
$\,\qt=\dzp\nnh/\hh\dz$, induced by the \lcc\ of\/ $\,g$, the latter is 
always flat, and the former is flat if\/ $\,\dz\hs$ is one\diml.
\end{enumerate}
\end{lemma}
\begin{proof}See \cite[Lemmas 2.1(ii) and 2.2]{derdzinski-roter-07}.
\end{proof}
For a $\,k\,$ times covariant tensor field $\,B\,$ on a \prd\ $\,(M,g)$, 
$\,k\ge1$, and a point $\,x\in M$, we denote by $\,\kb_x$ the subspace of 
$\,\txm\,$ formed by all vectors $\,v\,$ with 
$\,B_x(v,\,\cdot\,,\ldots,\,\cdot\,)=0$. Its orthogonal complement 
$\,(\kb_x)^\perp$ is the {\it image} of $\,B_x$, that is, the subspace of 
$\,\txm\,$ spanned by vectors $\,u\in \txm\,$ such that 
$\,g(u,\,\cdot\,)=B(\,\cdot\,,u_2,\dots,u_k)\,$ for some 
$\,u_2,\dots,u_k\in \txm$. If $\,B\,$ is parallel, the spaces $\,\kb_x$ and 
$\,(\kb_x)^\perp$ form parallel distributions $\,\kb\,$ and $\,(\kb)^\perp$ on 
$\,M$.

\begin{remark}\label{oldis}Given an \ecs\ \mf\ $\,(M,g)\,$ such that 
the \od\ $\,\dz\,$ is $\,2$\diml, we have 
$\,W\nh=\hs\ve\hskip1pt\om\nh\otimes\om\,$ for some $\,\ve=\pm\hs1\,$ and a 
parallel differential $\,2$-form $\,\hs\om\hs\,$ with 
$\,\hs\text{\rm rank}\,\hs\om\hs=2$, defined, at each point of $\,M$, only up 
to a sign. In addition,
\begin{equation}\label{owo}
\dz\,=\,(\text{\rm Ker}\,\hs\om)^\perp\nh,
\end{equation}
In fact, if $\,x\in M\,$ and $\,u,v,v\hh'$ are vector fields chosen so that 
$\,u_x\in\dz_x\smallsetminus\{0\}\,$ and $\,\hs\varOmega_x\ne0$, for 
$\,\hs\varOmega\hs=W(v,v\hh'\nnh,\,\cdot\,,\,\cdot\,)$, then the $\,2$-form 
$\,\varOmega_x$ is $\,\wedge$-di\-vis\-i\-ble by $\,\xi_{\hs x}$, where 
$\,\xi=g(u,\,\cdot\,)\,$ (cf.\ the definition of $\,\dz$). Thus, if $\,\dz\,$ 
is $\,2$\diml, the image of the Weyl tensor $\,W_{\nh x}$ acting on exterior 
$\,2$-forms is spanned by $\,\xi\wedge\xi\hs'\nnh$, where 
$\,\xi=g(u,\,\cdot\,)\,$ and $\,\xi\hs'\nh=g(u\hh'\nnh,\,\cdot\,)\,$ for 
any basis $\,u,u\hh'$ of $\,\dzx$. Since $\,W_{\nh x}$ acting on $\,2$-forms 
is self-ad\-joint, our claim follows, (\ref{owo}) being immediate as 
$\,(\text{\rm Ker}\,\hs\om)^\perp$ is the image of $\,\hs\om\hh$.
\end{remark}
Next, at any point $\,x\,$ of any \ecs\ \mf,
\begin{equation}\label{krp}
\mathrm{a)}\hskip12pt
(\kri_x)^\perp\subset\hs\dz_x\subset\hs\dzxp\subset\,\kri_x\hs,\hskip30pt
\mathrm{b)}\hskip12pt\dz\,\subset\,\hs\kw\hs,
\end{equation}
where $\,\dz\,$ is the \od\ and $\,\ri\,$ denotes the Ric\-ci tensor. Namely, 
the first inclusion in (a) follows from Lemma~\ref{csnpd}(ii) (as 
$\,(\kri_x)^\perp$ is the image of $\,\ri_x$), the second from 
Lemma~\ref{csnpd}(i), and the third from the first. For (b), we consider the 
two possible values $\,\rd\in\{1,2\}\,$ of the dimension of $\,\dz\,$ (see 
Lemma~\ref{csnpd}(i)). If $\,\rd=2$, we have $\,\dz=(\kw)^\perp\nnh$, that is, 
$\,\dz\,$ equals the image of $\,W\hs$ (which coincides with the image 
$\,(\text{\rm Ker}\,\hs\om)^\perp$ of $\,\hs\om\hh$, cf. (\ref{owo})), while 
$\,\dz\subset\dzp$ by Lemma~\ref{csnpd}(i). Now let $\,\rd=1$. As $\,W\hs$ and 
$\,\dzp$ are both parallel, it suffices to establish (b) at any fixed 
$\,x\in M\,$ with $\,\ri_x\ne0$. (Note that $\,g\,$ is not Ric\-ci-flat.) Now 
the image $\,(\kri_x)^\perp$ of $\,\ri_x$ is contained in $\,\kw_{\nh x}$ 
according to \cite[Theorem 8(c) on p.\ 22]{derdzinski-roter-78}, while 
$\,\dz_x=(\kri_x)^\perp$ by Lemma~\ref{csnpd}(ii), which yields (b).

\section{Proof of Theorem~\ref{pontr}, first part}\label{pfta}
The phrase `up to a factor' means, in this section, {\it up to a nonzero 
constant factor, which may depend on the dimensions involved.}

Given a pseu\-\hbox{do\hs-}Euclid\-e\-an inner product $\,\lr\,$ in 
an oriented \rvs\ $\,\mv\hs$ of even dimension $\,\er=2m$, let $\,\varTheta\,$ 
be the volume form, with $\,\varTheta(e_1,\dots,e_\er)=1\,$ for any 
pos\-i\-tive-o\-ri\-ent\-ed orthonormal basis $\,e_1,\dots,e_\er$ of 
$\,\mv\nnh$. We denote by $\,\hs\text{\rm Pf}\hs\,$ the {\it Pfaffian} 
function of $\,\lr$, assigning to an $\,m$-tuple of linear operators 
$\,S_j:\mv\to\mv\nnh$, which are all skew-ad\-joint relative to $\,\lr$, the 
value $\,s=\hs\text{\rm Pf}\,(S_1,\dots,S_m)\in\bbR\,$ such that 
$\,\zeta_1\nnh\wedge\ldots\wedge\zeta_m\nh=s\hs\varTheta$, with the 
$\,2$-forms $\,\zeta_j$ characterized by 
$\,\zeta_j(u,v)=\lg S_ju,v\rg\,$ for all $\,u,v\in\mv\nnh$.
\begin{lemma}\label{pfzer}For $\,S_j:\mv\to\mv\hs$ as above, 
$\,\hs\text{\rm Pf}\,(S_1,\dots,S_m)=0\,$ if\/ 
$\,\bigcap_{\hh j=1}^{\hh m}\text{\rm Ker}\,S_j\ne\{0\}$.
\end{lemma}
\begin{proof}Our $\,\zeta_j$ are pull\-backs to $\,\mv\hs$ of some $\,2$-forms 
in the quotient space $\,\mv\nnh/\mv\hh'\nnh$, where 
$\,\mv\hh'\nnh=\bigcap_{\hh j=1}^{\hh m}\text{\rm Ker}\,S_j$. Hence 
$\,\zeta_1\nnh\wedge\ldots\wedge\zeta_m\nh=0\,$ if 
$\,\dim\hs(\mv\nnh/\mv\hh'\hh)<m=\dim\mv\nh$.
\end{proof}
Given an oriented real vector bundle $\,\mathcal{E}\hs$ of fibre dimension 
$\,\er\ge1\,$ over a manifold $\,M$, let a pair $\,(\nabla,\gm)\,$ consist of 
a connection $\,\nabla\,$ and a $\,\nabla$\prl\ 
pseu\-\hbox{do\hs-}Riem\-ann\-i\-an fibre metric $\,\gm\,$ in $\,\mathcal{E}$. 
The {\it Euler form} of $\,(\nabla,\gm)\,$ then is the differential 
$\,\er$-form on $\,M\,$ equal to $\,0$, when $\,\er\hs$ is odd, and for even 
$\,\er\hs$ obtained, up to a factor, by skew-sym\-met\-riza\-tion of the 
$\,\er\hs$ times covariant tensor field that sends vector fields 
$\,v_1,\dots,v_\er$ to $\,\hs\text{\rm Pf}\,
(R^\nabla\nnh(v_1,v_2),\dots,R^\nabla\nnh(v_{\er-1},v_\er))$, cf.\ 
(\ref{cvo}), with $\,\hs\text{\rm Pf}\hs\,$ as above for 
$\,\mv\nh=\hs\mathcal{E}_x$, $\,x\in M$.

The Euler form of $\,(\nabla,\gm)\,$ is closed, and represents in cohomology 
the real Euler class of the oriented bundle $\,\mathcal{E}$. See 
\cite{avez-f}, \cite{chern}, \cite{law}, \cite{bell}.

Similarly, the real Pontryagin classes 
$\,p_i(\mathcal{E})\in H^{4i}(M,\bbR)\,$ of a real vector bundle 
$\,\mathcal{E}\,$ over a manifold $\,M\,$ are the cohomology classes of the 
Pontryagin forms of any connection $\,\nabla\,$ in $\,\mathcal{E}$, given by 
explicit formulae involving the curvature tensor $\,R=R^\nabla\nnh$. To prove 
vanishing of the Pontryagin forms (and classes) under some specific 
assumptions, one may instead use what we call here the {\it generating forms}, 
the cohomology classes of which form another set of generators for the 
Pontryagin algebra (the subalgebra of $\,H^*\nnh(M,\bbR)\,$ generated by all 
$\,p_i(\mathcal{E})$). The $\,i\hs$th generating form of $\,\nabla$, for any 
integer $\,i\ge1$, is the differential $\,4i$-form on $\,M\nh\,$ obtained, up 
to a factor, by skew-sym\-met\-riza\-tion of the $\,4i\,$ times covariant 
tensor field sending $\,v_1,\dots,v_{4i}$ to $\,\hs\text{\rm tr}\,
[R^\nabla\nnh(v_1,v_2)\circ\ldots\circ R^\nabla\nnh(v_{4i-1},v_{4i})]$, with 
$\,R^\nabla\nnh(v,v\hh'\hh)\,$ as in (\ref{cvo}). See \cite{chern-gcc}.

In the case where $\,\nabla\,$ is the \lcc\ of a \prd\ $\,(M,g)\,$ and 
$\,\xe=\tm$, we speak of the {\it Euler form\/} and {\it generating forms 
of\/} $\,(M,g)$.

Theorem~\ref{pontr} (minus the claim about $\,\pi_1M$) is immediate from 
the following lemma.
\begin{lemma}\label{eupoz}For an oriented \ecs\ manifold\/ $\,(M,g)\,$ of 
any dimension $\,n\ge4$, the Euler form and all Pontryagin forms are 
identically zero.
\end{lemma}
\begin{proof}We fix a point $\,x\in M$, an integer $\,i\,$ with 
$\,1\le i\le n/4$, where $\,n=\dim M$, and set 
$\,b=\hs\text{\rm tr}\,(B_1\circ\ldots\circ B_{2i})$, where 
$\,B_j=W_{\nh x}(u_{2j-1},u_{2j}):\txm\to\txm\,$ for linearly independent 
vectors $\,u_1,\dots,u_{4i}$ in $\,\txm$. (Notation as in (\ref{cvo}), with 
the Weyl tensor $\,W$ instead of $\,R$.) For vanishing of the Pontryagin 
forms, it suffices to prove that $\,b=0$, since, as shown by Avez 
\cite{avez-c}, in the definition of generating forms of $\,(M,g)\,$ one 
may replace the curvature tensor $\,R=R^\nabla$ by $\,W\nh$. We are thus 
allowed to choose $\,x\,$ at which the Ric\-ci tensor $\,\ri_x$ is nonzero: 
$\,W\hs$ is parallel, and hence so is the $\,i\hs$th generating form. For any 
fixed $\,v\hh'\nh\in\txm\,$ with $\,\ri_x(v\hh'\nnh,\,\cdot\,)\ne0$, 
Lemma~\ref{csnpd}(ii) implies that each $\,B_j$, when treated (with the aid of 
$\,g_x$) as a $\,2$-form at $\,x$, is $\,\wedge$-di\-vis\-i\-ble by 
$\,\ri_x(v\hh'\nnh,\,\cdot\,)$. In other words, 
$\,B_j:\txm\to\txm\,$ equals 
$\,g_x(w_j,\,\cdot\,)\hs v-\ri_x(v\hh'\nnh,\,\cdot\,)\hs w_j\,$ for some 
$\,w_j\in\txm\,$ and the unique $\,v\in\txm\,$ with 
$\,g_x(v,\,\cdot\,)=\ri_x(v\hh'\nnh,\,\cdot\,)$. Furthermore, 
$\,\dz_x\subset\kw_{\nh x}\subset\hs\text{\rm Ker}\,B_j$ and 
$\,\dz_x\subset\kri_x$ (see (\ref{krp})), so that, 
by Lemma~\ref{csnpd}(ii), $\,v\in\dz_x$ and $\,w_j\in\dzxp$ (as $\,w_j$ lies 
in the image $\,(\text{\rm Ker}\,B_j)^\perp$ of $\,B_j$). As $\,\dz\,$ is null 
(Lemma~\ref{csnpd}(i)), 
$\,B_1\nh\circ B_2=-\hs g_x(w_1,w_2)\ri_x(v\hh'\nnh,\,\cdot\,)\hs v$, 
$\,\hs\text{\rm tr}\,(B_1\nh\circ B_2)=0\,$ and 
$\,B_1\nh\circ B_2\nh\circ B_3=0\,$ if $\,i>1$, which implies that $\,b=0\,$ 
both for $\,i=1\,$ and $\,i>1$.

Now let $\,n\,$ be even. Given $\,x\in M$, we set 
$\,s=\hs\text{\rm Pf}\,(S_1,\dots,S_m)$, where $\,m=n/2\,$ and 
$\,S_j=R_x(e_{2j-1},e_{2j}):\txm\to\txm\,$ for a basis $\,e_1,\dots,e_n$ of 
$\,\txm\,$ containing a basis of $\,\dzxp$ (where $\,\dz\,$ is the \od). To 
obtain vanishing of the Euler form, we need to show that $\,s=0$. First, 
$\,s=0\,$ if, for each $\,j\in\{1,\dots,m\}$, at least one of the vectors 
$\,e_{2j-1},e_{2j}$ lies in $\,\dzxp\nnh$. In fact, Lemma~\ref{rwcod}(a) gives 
$\,R_x(v,u\hh'\nnh,u,\,\cdot\,)=0\,$ whenever $\,u\in\dz_x$, $\,v\in\dzxp$ and 
$\,u\hh'\nh\in\txm$, as $\,\dz_x\subset\kw_{\nh x}$ and 
$\,(g\wedge\ri)_x(v,u\hh'\nnh,u,\,\cdot\,)=0\,$ by (\ref{krp}); thus, 
$\,\bigcap_{\hh j=1}^{\hh m}\text{\rm Ker}\,S_j$ contains the subspace 
$\,\dz_x\ne\{0\}\,$ (cf.\ Lemma~\ref{csnpd}(i)), and Lemma~\ref{pfzer} shows 
that $\,s=0$.

In the remaining case, $\,e_{2j-1},e_{2j}\in\dzxp$ for some 
$\,j\in\{1,\dots,m\}$. Namely, if $\,\rd\,$ denotes the dimension of $\,\dz$, 
then $\,\dzp$ is $\,(2m-\rd)$\diml, with $\,\rd\le2\le m\,$ by 
Lemma~\ref{csnpd}(i). Among $\,e_1,\dots,e_{2m}$ there are 
$\,2m-\rd\ge m\,$ elements of $\,\dzxp\nnh$, so that one of the 
$\,m\,$ sets $\,\varSigma_j=\{e_{2j-1},e_{2j}\}\,$ must be contained in 
$\,\dzxp$ (or else  $\,\varSigma_j\nh\cap\dzxp$ would, for each 
$\,j=1,\dots,m$, have exactly one element, leading to a case which we already 
excluded).

Now that $\,e_{2j-1},e_{2j}\in\dzxp$ for some $\,j$, 
Lemma~\ref{csnpd}(iii) gives $\,S_j=0$, and hence $\,s=0$.
\end{proof}

\section{Proof of Theorem~\ref{parvf}}\label{pote}
In any \ecs\ \mf\ $\,(M,g)\,$ such that the \od\ $\,\dz\,$ is one\diml, 
setting $\,\qt=\dzp\nnh/\hh\dz$, one has a vec\-tor-bun\-dle morphism
\begin{equation}\label{phi}
\varPhi:(\dz^*)^{\otimes2}\to\hs(\qt^*){}^{\otimes2}
\end{equation}
defined as follows. Given $\,x\in M\,$ and $\,\lambda,\lambda'\in\dzxa$, we 
declare 
$\,\varPhi_x(\lambda\otimes\lambda'\hh):\qt_x\times\qt_x\to\bbR\,$ to be the 
symmetric bilinear form sending the cosets $\,v+\dz_x$ and 
$\,v\hh'\nh+\dz_x$ of vectors $\,v,v\hh'\nh\in\dzxp$ to 
$\,W_{\nh x}(v,u,u\hh'\nnh,v\hh'\hs)$, where $\,u,u\hh'\nh\in \txm\,$ are any 
vectors with $\,\lambda=g_x(u,\,\cdot\,)\,$ and 
$\,\lambda'\nh=g_x(u\hh'\nnh,\,\cdot\,)\,$ on $\,\dz_x$. Note that, as 
$\,\dz_x\subset\kw_{\nh x}$ by (\ref{krp}.b), the value $\,W_{\nh x}(v,u,u\hh'\nnh,v\hh'\hs)\,$ 
depends just on the $\,\dz_x$-co\-sets, rather than the vectors $\,v,v\hh'$ 
themselves, while, by Lemma~\ref{csnpd}(iii), 
$\,W_{\nh x}(v,u,u\hh'\nnh,v\hh'\hs)\,$ is not affected by how $\,u\,$ and 
$\,u\hh'$ were chosen: two such choices of either vector differ by an 
element of $\,\dzxp\nnh$.
\begin{remark}\label{paral}If $\,(M,g)\,$ is \ecs\ and $\,\dz\,$ is one\diml, 
the \lcc\ of $\,g\,$ induces flat connections in both bundles 
$\,(\dz^*)^{\otimes2}$ and $\,(\qt^*){}^{\otimes2}$ (by 
Lemma~\ref{csnpd}(iv)), and the morphism $\,\varPhi\,$ is
\begin{enumerate}
  \def\theenumi{{\rm\alph{enumi}}}
\item parallel relative to those connections,
\item nonzero, and hence injective, at every point $\,x\in M$.
\end{enumerate}
Here (a) states that $\,\varPhi\,$ is parallel as a section of the bundle 
$\,\hs\text{\rm Hom}\hs((\dz^*)^{\otimes2}\nnh,\hs(\qt^*){}^{\otimes2}\hh)\,$ 
with the induced flat connection, or, equivalently, that the 
$\,\varPhi$-im\-age of any parallel local section of $\,(\dz^*)^{\otimes2}$ 
is parallel in $\,(\qt^*){}^{\otimes2}\nnh$.

Assertion (a) is obvious from naturality of $\,\varPhi$, since $\,W\hs$ is 
parallel. To verify (b), note that if we had $\,W_{\nh x}(v,u,u,v\hh'\hs)=0\,$ 
for a fixed $\,u\in\txm\smallsetminus\dzxp$ and all 
$\,v,v\hh'\nh\in\dzxp\nnh$, the components of $\,W_{\nh x}$ in a basis 
consisting of $\,u\,$ and a basis of $\,\dzxp$ would all vanish (by 
Lemma~\ref{csnpd}(iii)), even though $\,(M,g)\,$ is not conformally flat.
\end{remark}
\begin{remark}\label{fibme}Given an \ecs\ \mf\ $\,(M,g)$, let $\,\dz\,$ and 
$\,\qt\,$ be as in (\ref{phi}). Since $\,\dz\,$ is the 
$\,g$-null\-space subbundle of $\,\dzp$ (cf.\ Lemma~\ref{csnpd}(i)), the 
metric $\,g$, restricted to $\,\dzp\nnh$, descends to a 
pseu\-\hbox{do\hs-}Riem\-ann\-i\-an fibre metric $\,\gm\,$ on $\,\qt$. 
Clearly, $\,\gm$ is parallel relative to the connection induced by the \lcc\ 
$\,\nabla\,$ of $\,g$. Being $\,\nabla$\prl, $\,\dzp$ is integrable and has 
totally geodesic leaves, and the \lcc\ of $\,g\,$ induces on each leaf a 
tor\-sion\-free connection, which is flat in view of Lemma~\ref{csnpd}(iii).

If the sign pattern of $\,g\,$ is $\,(i_-,i_+)$, with $\,i_-$ minuses and 
$\,i_+$ pluses, then $\,\gm\,$ has the sign pattern $\,(i_--\rd,i_+-\rd)$, 
where $\,\rd\,$ the dimension of the distribution $\,\dz$. This is clear if 
one chooses a subspace $\,\mv\hs$ of $\,\dzxp$ with 
$\,\dzxp\nh=\dz_x\oplus\mv\nnh$, for any $\,x\in M$, and notes that 
$\,\txm=\mv\nnh\oplus\mv^\perp\nnh$, while $\,\mv\hs$ and $\,\mv^\perp$ have 
the sign patterns equal to that of $\,\gm\,$ and, respectively, 
$\,(\rd,\rd)\,$ (as $\,\dz_x\subset\mv^\perp$).
\end{remark}
\begin{proof}[Proof of Theorem~{\rm\ref{parvf}}]Let $\,(M,g)\,$ be \ecs\ and 
Lo\-rentz\-i\-an, and let $\,\rd\,$ be the dimension of $\,\dz$. Then 
$\,\rd=1\,$ by Lemma~\ref{csnpd}(i), since $\,\rd\le1\,$ due to the 
Lo\-rentz\-i\-an sign pattern \mpdp\hs. (In fact, $\,\txm\,$ contains 
$\,\mv\oplus\dz_x$, for any $\,x\in M\,$ and any co\-di\-men\-sion-one 
subspace $\,\mv\subset \txm\,$ on which $\,g_x$ is positive definite.)

The parallel injective morphism $\,\varPhi\,$ in (\ref{phi}) now gives rise to 
a fibre norm $\,|\hskip2.3pt|\,$ in the line bundle $\,\dz$, which is parallel 
(invariant under parallel transports). Namely, for $\,\hs x\in M$ and 
$\,u\in\dz_x\nh\smallsetminus\{0\}$, we set 
$\,|u|=|\hs\varPhi_x(\lambda\otimes\lambda)|^{-1/2}\nnh$, where 
$\,\lambda\in\dzxa$ is chosen so that $\,\lambda(u)=1$, and the latter 
$\,|\hskip2.3pt|\,$ is the fibre norm in $\,(\qt^*){}^{\otimes2}$ 
corresponding to the fibre metric $\,\gm\,$ in $\,\qt$. Note that $\,\gm\,$ is 
positive definite as $\,\rd=1\,$ (see Remark~\ref{fibme}). Since a 
$\,|\hskip2.3pt|$-unit section of $\,\dz$ is parallel, this proves 
Theorem~\ref{parvf}.
\end{proof}

\section{Proof of Theorem~\ref{bdcir}}\label{potb}
Let an \ecs\ \mf\ $\,(M,g)\,$ satisfy one of the conditions
\begin{enumerate}
  \def\theenumi{{\rm\roman{enumi}}}
\item $M\,$ is simply connected and the \od\ $\,\dz\,$ is one\diml, or
\item $g\,$ is Lo\-rentz\-i\-an and $\,\dz\,$ is trivial as a real line 
bundle, cf.\ Theorem~\ref{parvf}.
\end{enumerate}
Then there exist functions $\,\psi,\phi:M\to\bbR\,$ and a vector field 
$\,u\,$ on $\,M\,$ such that
\begin{enumerate}
  \def\theenumi{{\rm\alph{enumi}}}
\item $u\,$ is parallel, nonzero, and spans $\,\dz$,
\item the $\,1$-form $\,\xi=g(u,\,\cdot\,)\,$ is parallel, the Ric\-ci tensor 
$\,\ri\,$ equals $\,\psi\,\xi\otimes\xi$, and $\,d\psi=\phi\,\xi$,
\item $\phi\,$ is nonconstant if $\,M\,$ is compact.
\end{enumerate}
Still assuming (i) or (ii), we define a vec\-tor-bun\-dle morphism 
$\,A:\qt\to\qt\hs$ by \hbox{requiring that} $\,\gm_x(A_x\eta,\,\cdot\,)
=[\varPhi_x(\lambda\otimes\lambda)](\eta,\,\cdot\,)\,$ for $\,x\in M\,$ and 
$\,\eta\in\qt_x$ with $\,\lambda\in\dz_x^*$ such that $\,\lambda(u_x)=1$. 
(Notation of Lemma~\ref{csnpd}(iv), Remark~\ref{fibme} and (\ref{phi}).) Then
\begin{enumerate}
  \def\theenumi{{\rm\alph{enumi}}}
\item[(d)] $A\,$ is $\,\nabla$\prl\ as a section of 
$\,\hs\text{\rm Hom}\hs(\qt,\qt)$, nonzero, self-ad\-joint relative to 
$\,\gm\,$ and traceless at every point,
\item[(e)] in the case where $\,A_x:\qt_x\to\qt_x$ has 
$\,n-2=\dim\qt_x$ distinct eigenvalues at some/every point 
$\,x\in M$, for $\,n=\dim M\ge4$, the bundle $\,\qt\hs$ over $\,M\,$ is 
an orthogonal direct sum of $\,\nabla$\prl\ real-line subbundles.
\end{enumerate}
Under the assumption (i), there exists a function $\,t:M\to\bbR\,$ such that, 
for $\,\xi=g(u,\,\cdot\,)$,
\begin{enumerate}
  \def\theenumi{{\rm\alph{enumi}}}
\item[(f)] $\xi=\hh dt\,$ (or, equivalently, $\,u=\nabla t$) and $\,\psi\,$ is, 
locally, a function of $\,t$.
\end{enumerate}
In fact, $\,u\,$ in (a) exists in view of Lemma~\ref{csnpd}(iv) and 
Theorem~\ref{parvf}, while (b), for some $\,\psi\,$ and $\,\phi$, follows from 
Lemma~\ref{csnpd}(ii), since $\,\nabla\nh\ri=d\psi\otimes\xi\otimes\xi\,$ is 
totally symmetric by Lemma~\ref{rwcod}(b). Now, if $\,\phi\,$ were constant 
and $\,M\,$ compact, 
$\,d\psi\,$ would be parallel (as $\,\xi\,$ is), and so, being zero 
somewhere, $\,d\psi\,$ would vanish identically. However, $\,\psi\,$ is 
nonconstant, since $\,\ri\,$ cannot be parallel: $\,g\,$ is conformally 
symmetric but not locally symmetric. This contradiction proves (c). Next, 
(d) holds in view of Remark~\ref{paral} and 
Theorem~\ref{parvf}, with trace\-less\-ness of $\,A\,$ due to vanishing of 
the contractions of $\,W\nh$. Assertion (e) is now immediate, the subbundles 
in question being the eigenspace bundles of $\,A$. Finally, $\,\xi\,$ is 
parallel, and hence closed, so that (b) implies (f).
\begin{proof}[Proof of Theorem~{\rm\ref{bdcir}}]Let $\hs(M,g)\hs$ be a compact 
\ecs\ Lo\-rentz\-i\-an \mf. Theorem~\ref{parvf} allows us to assume that 
$\,(M,g)\,$ admits a global parallel vector field $\,u\,$ spanning the 
one\diml\ \npd\ $\,\dz$. Condition (ii) above is therefore satisfied, 
which implies (b), while the function $\,\phi\,$ in (b) is nonconstant 
by (c). Our assertion is now immediate from Lemmas~\ref{clext} 
and~\ref{csnpd}(iii).
\end{proof}
\begin{remark}\label{concl}Let $\,(M,g)\,$ be any compact \ecs\ 
Lo\-rentz\-i\-an \mf\ such that the \od\ $\,\dz\,$ is trivial as a real line 
bundle. Choosing $\,\xi,\phi,\psi\,$ as in (a) -- (c), and $\,t\,$ as in (f) 
(where, for $\,t\,$ to exist, we use instead of $\,(M,g)\,$ its universal 
covering \mf\ $\,(\hm,\hg\hs)$), we see that $\,\xi,\phi,\psi\,$ and $\,t\,$ 
satisfy all the hypotheses of Lemma~\ref{clext}. Consequently, {\it they 
satisfy the conclusions of Lemma~{\rm\ref{clext}} as well}. This proves 
the claims immediately following Theorem~\ref{nolor} in the Introduction.
\end{remark}

\section{Examples}\label{xpls}
Suppose that we are given a nonconstant $\,C^\infty$ function 
$\,\fh:\bbR\to\bbR\hs$, a \rvs\ $\,\mv\hs$ of dimension $\,n-2\ge2\,$ with a 
pseu\-\hbox{do\hs-}Euclid\-e\-an inner product $\,\lr$, and a nonzero 
traceless linear operator $\,A:\mv\to\mv\nh$, self-ad\-joint relative to 
$\,\lr$. Following \cite{roter}, we use such data to define a \prc\ 
$\,\hg=\kx\,dt^2\hs+\,dt\,ds\,+\,\vh\,$ on the manifold 
$\,\hm=\rto\nh\times\mv\nh$, \feic\ to $\,\rn\nnh$, where products of 
differentials stand for symmetric products, $\,t,s\,$ are the Cartesian 
coordinates on the $\,\rto$ factor, $\,\vh\,$ denotes the pull\-back to 
$\,\hm\,$ of the flat \prc\ on $\,\mv$ corresponding to the inner product 
$\,\lr$, and $\,\kx:\hm\to\bbR\,$ is given by 
$\,\kx(t,s,v)=\fh(t)\hh\lg v,v\rg+\lg Av,v\rg$.

Let $\,\xe$ be the vector space of all $\,C^\infty$ solutions 
$\,u:\bbR\to\mv\,$ to the differential equation 
$\,\ddot u(t)=\fh(t)\hh u(t)+Au(t)$, and let $\,\bbP\,$ be the additive group 
of all $\,p\in\bbR\,$ with $\,\fh(t+p)=\fh(t)$ for every real $\,t$. The set 
$\,\gp=\bbP\times\bbR\times\xe\,$ 
has a unique group structure such that the formula 
$\,(p,q,u)\cdot(t,s,v)=(t+p,\hs s+q-\lg\dot u(t),2v+u(t)\rg,\hs v+u(t))$, for 
$\,(p,q,u)\in\gp\,$ and $\,(t,s,v)\in\hm=\rto\nh\times\mv\nh$, describes a 
group action of $\,\gp\hs$ on $\,\hm$. (See 
\cite[Section 2]{derdzinski-roter}.)
\begin{lemma}\label{ecsrr}For any choice of the above data 
$\,\fh,n,\mv\nnh,\lr\,$ and\/ $\,A$,
\begin{enumerate}
  \def\theenumi{{\rm\roman{enumi}}}
\item the metric\/ $\,\hg\,$ is \ecs,
\item the sign pattern of\/ $\,\hg\,$ arises from that of\/ $\,\lr\,$ by 
adding one plus and one minus,
\item the group\/ $\,\gp\hs$ acts on\/ $\,(\hm,\hg\hs)\,$ by isometries,
\item if\/ $\,n=4\,$ and the metric $\,\hg\,$ is Lo\-rentz\-i\-an,
\begin{enumerate}
  \def\theenumi{{\rm\alph{enumi}}}
\item[a)] $\gp\hs$ is a subgroup of finite index in the full isometry group 
of\/ $\,(\hm,\hg\hs)$,
\item[b)] $(\hm,\hg\hs)\hs$ is not the universal covering space of any 
compact \prd.
\end{enumerate}
\end{enumerate}
\end{lemma}
\begin{proof}For (i), see \cite[Theorem 3]{roter} or 
\cite[Lemma 2.1]{derdzinski-roter}, while (ii) is obvious, and (iii) is 
immediate from \cite[Lemma 2.2]{derdzinski-roter}.

Generally, the index $\,\hs{\rm ind}\hs(\gp\hs'\nnh,\gp)\,$ of a subgroup 
$\,\gp\,$ in a group $\,\gp\hs'$ is the cardinality of the quotient set 
$\,\gp\hs'\nnh/\hs\gp\,$ consisting of all left co\-sets of $\,\gp\,$ in 
$\,\gp\hs'\nnh$. If $\,\hp\hs'\nh\subset\gp\hs'$ is a subgroup such that 
$\,\gp\hs'\nh=\hp\hs'\gp$, then, for $\,\hp=\gp\cap\hp\hs'\nnh$, the 
inclusion mapping $\,\hp\hs'\nh\to\gp\hs'$ clearly induces a bijection 
$\,\hp\hs'\nnh/\hs\hp\,\to\,\gp\hs'\nnh/\hs\gp$, and so 
$\,\hs{\rm ind}\hs(\gp\hs'\nnh,\gp)=\hs{\rm ind}\hs(\hp\hs'\nnh,\hp)$. 
Here are two special cases in which groups $\,\gp\hs'\nnh,\gp\,$ and 
$\,\hp\,$ satisfy the above assumption, and hence the conclusion:
\begin{enumerate}
  \def\theenumi{{\rm\roman{enumi}}}
\item[(I)] $\hp\hs'\nh=\hs\text{\rm Ker}\,\varPi\,$ for a group homomorphism 
$\,\varPi:\gp\hs'\nh\to\kp\,$ with $\,\varPi(\gp)=\kp\hs$,
\item[(II)] $\gp\hs'$ acts from the left on a set $\,Y\nh$, the action 
restricted to $\,\gp\hs$ is transitive, and $\,\hp\hs'$ is the isotropy 
subgroup of $\,\gp\hs'$ at some $\,y\in Y\nh$.
\end{enumerate}
Let $\,\hg\,$ now be Lo\-rentz\-i\-an, and let $\,\gp\hs'$ denote the group 
of those isometries of $\,\nh(\hm,\hg\hs)\hs$ 
which preserve the $\,1$-form $\,dt$. The Ric\-ci tensor of $\,\hg\,$ is given 
by $\,\ri=(2-n)\hs\fh(t)\,dt\otimes dt$, and $\,dt\,$ is parallel. (See 
\cite[p.\ 93]{roter}, where the sign convention for $\,\ri\,$ is the opposite 
of ours.) Thus, by Lemma~\ref{csnpd}(ii), the \od\ $\,\dz\,$ is spanned by the 
null parallel vector field $\,u=\nabla t$. Since $\,u\,$ can be naturally 
normalized with the aid of a parallel fibre norm in $\,\dz\,$ (see the end of 
Section~\ref{pote}), isometries of $\,(\hm,\hg\hs)\,$ leave $\,dt\,$ invariant 
up to a sign, so that the full isometry group of $\,(\hm,\hg\hs)\,$ contains 
$\,\gp\hs'$ as a subgroup of index at most $\,2$, and (iv-a) will follow if we 
show that $\,\hs{\rm ind}\hs(\gp\hs'\nnh,\gp)\,$ is finite.

As $\,\gp\hs'$ preserves $\,dt\,$ and $\,\ri=(2-n)\hs\fh(t)\,dt\otimes dt$, we 
have $\,t\circ\alpha=t+\varPi(\alpha)\,$ for some homomorphism 
$\,\varPi:\gp\hs'\nh\to\bbR\,$ and all $\,\alpha\in\gp\hs'\nnh$, while the 
function $\,\fh(t):\hm\to\bbR\,$ is $\,\gp\hs'\nh$\inv. Thus, 
$\,\varPi(\gp\hs'\hs)\,$ coincides with the additive group $\,\bbP$ defined 
earlier in this section, and so 
$\,\hs{\rm ind}\hs(\gp\hs'\nnh,\gp)
=\hs{\rm ind}\hs(\text{\rm Ker}\,\varPi,\gp\cap\hs\text{\rm Ker}\,\varPi)$, 
in view of (I) above for $\,\kp=\bbP\,$ and $\,\varPi:\gp\hs'\nh\to\bbP$.
Next, for any fixed $\,t\in\bbR\hs$, the action of 
$\,\hs\text{\rm Ker}\,\varPi\,$ leaves the affine subspace 
$\,\hmt=\{t\}\times\bbR\times\mv\hs$ of $\,\hm\,$ invariant, and the 
restriction of this action to 
$\,\gp\cap\hs\text{\rm Ker}\,\varPi$ is transitive on $\,\hmt$, since it 
consists of affine transformations realizing all translational parts in 
$\,\{0\}\times\bbR\times\mv\nh$. Consequently, (II) gives 
$\,\hs{\rm ind}\hs(\gp\hs'\nnh,\gp)=\hs{\rm ind}\hs(\hp\hs'\nnh,\hp)$, 
where $\,\hp\hs'$ is the isotropy subgroup of $\,\gp\hs'$ at any fixed 
$\,x\in\hm$, and $\,\hp=\gp\cap\hp\hs'\nnh$.

On the other hand, $\,\hs{\rm ind}\hs(\hp\hs'\nnh,\hp)\le4\,$ if $\,n=4$. 
Namely, the infinitesimal action of $\,\hp\hs'$ on 
$\,\txhm=\rto\nh\times\mv\hs$ preserves $\,dt_x$ and the vector 
$\,u_x=(0,1/2,0)$, along with the subspace $\,u_x^\perp\nh=\dzxp\nnh$. Hence 
the action descends to $\,\qt_x=\dzxp\nnh/\dz_x$, where it preserves the inner 
product $\,\gm_x$ (see Remark~\ref{fibme}) and commutes with the operator 
$\,A_x:\qt_x\to\qt_x$ which, under the obvious identification 
$\,\qt_x\approx\hs\mv\nnh$, coincides with our $\,A:\mv\to\mv\nh$. (See 
\cite[the description of $\,W\hs$ on p.\ 93]{roter}.) The differentials at 
$\,x\,$ of elements of $\,\hp\hs'\nnh$, acting on 
$\,\txhm=\rto\nh\times\mv\nnh$, thus have the form 
$\,(t,s,v)\mapsto(t,s+\varphi(v),Lv)$, where $\,L:\mv\to\mv\hs$ is a linear 
isometry commuting with $\,A$, and $\,\varphi\in\mv\hh^*\nnh$. As elements of 
the subgroup $\,\hp\,$ realize all $\,\varphi\in\mv\hh^*\nnh$, and have 
$\,L=\hs\text{\rm Id}\hh$, it follows that 
$\,\hs{\rm ind}\hs(\hp\hs'\nnh,\hp)\le4\,$ (see Remark~\ref{cmiso} below), 
which yields (iv-a). 

Finally, to prove (iv-b), we may suppose that, on the contrary, some group 
$\,\Gm\,$ of isometries of $\,(\hm,\hg\hs)\,$ acts on $\,\hm\,$ properly 
dis\-con\-tin\-u\-ous\-ly, producing a compact quotient manifold. The same 
is then true for the subgroup $\,\Gm\cap\gp\,$ of $\,\Gm\,$ (as 
$\,\Gm\cap\gp\,$ is of finite index in $\,\Gm$, by (iv-a)), which in turn 
contradicts \cite[Theorem~7.3]{derdzinski-roter}. Note that periodicity 
of $\,\fh\,$ as a function of $\,t$, required in \cite{derdzinski-roter}, 
follows from Lemma~\ref{clext}(b), cf.\ Remark~\ref{concl}.
\end{proof}
\begin{remark}\label{cmiso}For a nonzero traceless self-ad\-joint linear 
endomorphism $\,A\,$ of a pseu\-\hbox{do\hs-}Euclid\-e\-an plane $\,\mv\nh$, 
there exist at most four linear isometries $\,L:\mv\to\mv\hs$ commuting with 
$\,A$. This is clear when $\,A\,$ is di\-ag\-o\-nal\-izable, since $\,L\,$ 
must then send an orthonormal basis $\,(v,w)\,$ diagonalizing $\,A\,$ to 
$\,(\pm\hs v,\pm\hs w)\,$ or $\,(\pm\hs v,\mp\hs w)$. On the other hand, if a 
linear isometry $\,L\,$ commutes with $\,A\,$ and $\,A\,$ is 
non-di\-ag\-o\-nal\-izable (so that $\,\mv\hs$ is Lo\-rentz\-i\-an), we have 
$\,L=\pm\hs\text{\rm Id}\hh$. In fact, let $\,L\ne\pm\hs\text{\rm Id}\hh$. 
The two null lines in $\,\mv\hs$ are interchanged by $\,L$ (if they were 
preserved, $\,L\,$ would be di\-ag\-o\-nal\-izable, implying the same for 
$\,A$). However, choosing a basis $\,(v,w)\,$ of null vectors with $\,Lv=w\,$ 
and $\,Lw=v$, we would then again diagonalize $\,L\,$ (and hence $\,A$), 
this time with the eigenvectors $\,v\pm w$.
\end{remark}

\section{A classification theorem}\label{clth}
In the following theorem, $\,t\,$ denotes any fixed function $\,\hm\to\bbR\,$ 
such that $\,u=\nabla t\,$ is a global parallel vector field spanning the 
\od\ $\,\dz$. Such $\,t\,$ exists according to (f) in Section~\ref{potb}, 
cf.\ Theorem~\ref{parvf}.
\begin{theorem}\label{class}Let\/ $\,(\hm,\hg\hs)\,$ be a simply connected 
\ecs\ Lo\-rentz\-i\-an \mf\ of dimension $\,n\ge4\,$ such that the leaves of 
the parallel distribution $\,\dzp$ are all complete and the function 
$\,t:\hm\to\bbR\,$ satisfies condition\/ {\rm(c)} in Lemma~{\rm\ref{clext}}. 
Then, up to an isometry, $\,(\hm,\hg\hs)\,$ is one of the manifolds 
constructed in Section\/~{\rm\ref{xpls}}.
\end{theorem}
Our proof of Theorem~\ref{class}, given in Section~\ref{potc}, uses the facts 
presented below.

Let $\,(\hm,\hg\hs)\,$ be a simply connected \ecs\ \mf\ such that the \od\ 
$\,\dz\,$ is one\diml. If $\,v,v\hh'$ are sections of $\,\dzp$ and 
$\,u\,$ is a fixed nonzero parallel section of $\,\dz$, while $\,u\hh'$ is 
any vector field, then
\begin{equation}\label{ruv}
R(u\hh'\nnh,v)\hh v\hh'\,
=\,g(u\hh'\nnh,u)\hs[\hs\gm(A\uv,\uvp\hs)+\fh g(v,v\hh'\hh)\hh]\hs u\hs,
\end{equation}
where $\,\fh:\hm\to\bbR\,$ is given by $\,\fh=(2-n)^{-1}\psi$, for 
$\,n=\dim\hm\ge4$, with $\,\psi\,$ and $\,A\,$ defined as in 
Section~\ref{potb}, and $\,\uv\,$ denotes the image of $\,v\,$ under the 
quo\-tient-pro\-jec\-tion morphism $\,\dzp\nnh\to\qt=\dzp\nnh/\hh\dz$. (Thus, 
$\,\fh\,$ is characterized by $\,\ri=(2-n)\hs\fh\,\xi\otimes\xi$, for the 
parallel $\,1$-form $\,\xi=g(u,\,\cdot\,)=\hh dt$.)

In fact, by Lemma~\ref{csnpd}(iii), $\,R(u\hh'\nnh,v)\hh v\hh'$ is orthogonal 
to $\,\dzp\nnh$, and hence equals a function times $\,u$. Since both 
sides of (\ref{ruv}) are linear in $\,u\hh'\nnh$, (\ref{ruv}) will follow if, 
under the assumption $\,g(u\hh'\nnh,u)=1$, applying 
$\,g(u\hh'\nnh,\,\cdot\,)\,$ to both sides we obtain the same value. This 
last conclusion is in turn immediate from Lemma~\ref{rwcod}(a) combined with 
the definition of $\,A\,$ in Section~\ref{potb} and the relation 
$\,\ri=(2-n)\hs\fh\,\xi\otimes\xi$. (By (\ref{krp}.a), 
$\,\ri\hh(v,\,\cdot\,)=\ri\hh(v\hh'\nnh,\,\cdot\,)=0$.)
\begin{remark}\label{param}Let $\,(\hm,\hg\hs)\,$ satisfy the assumptions of 
Theorem~\ref{class}. We say that a curve $\,I\to\hm$, defined on an interval 
$\,I\subset\bbR\hs$, is {\it parametrized by the function\/} 
$\,t:\hm\to\bbR\,$ (chosen at the beginning of this section) if $\,t\,$ sends 
the image of the curve \feicly\ onto $\,I$. Such a curve may be written as 
$\,I\ni t\mapsto y(t)\in\hm$, with $\,t\,$ serving as the parameter.
\begin{enumerate}
  \def\theenumi{{\rm\alph{enumi}}}
\item Up to an affine re-pa\-ram\-e\-tri\-za\-tion, every geodesic in 
$\,(\hm,\hg\hs)$, not tangent to the distribution $\,\dzp\nnh$, is 
parametrized by the function $\,t$.
\item If a curve $\,\bbR\ni t\mapsto y(t)\in\hm\,$ is parametrized by the 
function $\,t$, then so is every curve $\,t\mapsto x(t,s)\,$ in the variation 
given by $\,x(t,s)=\hs\exp_{\hs y(t)}sw(t)$, where 
$\,\bbR\ni t\mapsto w(t)\in\dz_{y(t)}^\perp$ is any vector field along the 
original curve, tangent to $\,\dzp\nnh$.
\item For any curve $\,\bbR\ni t\mapsto y(t)\in\hm\,$ is parametrized by the 
function $\,t$, we have $\,g(\dot y,u)=1$, for $\,u=\nabla t$, and 
$\,\nabla_{\!\dot y}\dot y\,$ is tangent to $\,\dzp\nnh$.
\end{enumerate}
In fact, (a) and (b) follow since $\,\nabla dt=0$, so that $\,t\,$ restricted 
to any geodesic is an affine function of the geodesic parameter, while the 
leaves of $\,\dzp$ are totally geodesic (Remark~\ref{fibme}), and $\,t\,$ 
is constant on each leaf as $\,\dzp\nh=\hs\text{\rm Ker}\,dt$. Now (c) 
is immediate as $\,\dzp\nh=u^\perp$ and $\,u=\nabla t\,$ is parallel: 
differentiating $\,g(\dot y,u)=1$, we get $\,g(\nabla_{\!\dot y}\dot y,u)=0$.
\end{remark}
The following lemma is a crucial step in proving Theorem~\ref{class}.
\begin{lemma}\label{compl}Under the assumptions of Theorem~{\rm\ref{class}}, 
$\,(\hm,\hg\hs)\,$ is complete.
\end{lemma}
\begin{proof}Using (c) in Lemma~\ref{clext}, we may fix a curve 
$\,\bbR\ni t\mapsto y(t)\in\hm\,$ parametrized by the function $\,t\,$ (cf.\ 
Remark~\ref{param}), and consider the differential equation
\begin{equation}\label{ode}
\nabla_{\!\dot y}\nabla_{\!\dot y}w+R(\dot y,w)\hh\dot y
+\nabla_{\!\dot y}\dot y\,=\,-\hs Q(w)\hh u/4
\end{equation}
imposed on vector fields $\,w\,$ along the curve which are tangent to 
$\,\dzp\nnh$. Here $\,u=\nabla t\,$ and 
$\,Q(w)=3\hs[\gm(A\uw,\hs\uw)]\dot{\,}\nh
+3\fh\hs[g(w,w)]\dot{\,}\nh+2\dot\fh g(w,w)$, with $\,\fh,A,\uw\,$ as in 
(\ref{ruv}), and $\,[\hskip4pt]\dot{\,}\nh=\hh d/dt$. (Thus,
$\,\dot\fh=\hh\dfh(y(t))/dt$.) By Remark~\ref{param}(c), both sides of 
(\ref{ode}) are tangent to $\,\dzp\nnh$, that is, orthogonal to $\,u$, as 
$\,\nabla u=0\,$ and so $\,R(\,\cdot\,,\,\cdot\,,\,\cdot\,,u)=0$.

Every solution $\,w\,$ to (\ref{ode}) can be defined on the whole real line. 
Namely, this is true, due to linearity, for solutions $\,\widetilde w\,$ 
(tangent to $\,\dzp$) of the equation 
$\,\nabla_{\!\dot y}\nabla_{\!\dot y}\widetilde w
+R(\dot y,\widetilde w\hs)\hh\dot y+\nabla_{\!\dot y}\dot y=0$. Using any such 
$\,\widetilde w\,$ and any function $\,\my:\bbR\to\bbR\,$ with the second 
derivative $\,\ddot\my=Q(\widetilde w\hs)/4$, we now get a solution 
$\,w=\widetilde w-\my u\,$ to (\ref{ode}), defined on $\,\bbR\hs$. (Note 
that $\,Q(w)=Q(\widetilde w\hs)$, and $\,R(\dot y,u)\hh\dot y=0\,$ since 
$\,\nabla u=0$.)

Any solution $\,w\,$ to (\ref{ode}), defined on $\,\bbR\hs$, leads to the 
variation of curves in $\,\hm\,$ given by $\,x(t,s)=\hs\exp_{\hs y(t)}sw(t)$. 
Let $\,v\,$ be the vector field along the variation such that $\,v_s=0\,$ for 
all $\,(t,s)$, and $\,v=\nabla_{\!\dot y}\dot y\,$ at $\,s=0\,$ (notation as 
in (\ref{xts})). We have
\begin{equation}\label{xtt}
\mathrm{i)}\hskip9ptx_{tt}\,+\,(s-1)[v-s\hs Q(x_s)\hh u/4\hh]\,=\,0\hs,
\hskip29pt\mathrm{ii)}\hskip9pt[Q(x_s)]_s\,=\,\hs0\hs.
\end{equation}
where the subscripts now also stand for partial derivatives of functions of 
$\,(t,s)$, and $\,Q(x_s)=3\hs[\gm(A\uxs,\hs\uxs)]_t
+3\fh\hs[g(x_s,x_s)]_t+2\fh_t\hs g(x_s,x_s)$. Before proving (\ref{xtt}), note 
that
\begin{equation}\label{uxo}
\mathrm{a)}\hskip7ptg(x_t,u)=1\hs,
\hskip9pt\mathrm{b)}\hskip7ptx_{ss}=0\hs,
\hskip9pt\mathrm{c)}\hskip7ptx_s,x_{tt}\text{\rm\ and\ }\,x_{st}=x_{ts}
\text{\rm\ are\ all\ tangent\ to\ }\,\dzp.
\end{equation}
In fact, (\ref{uxo}.a) and the claim about $\,x_{tt}$ in (\ref{uxo}.c) are 
immediate from clauses (b), (c) in Remark~\ref{param}. Next, $\,x_{ss}=0\,$ 
and $\,x_s$ is tangent to $\,\dzp\nnh$, as $\,\dzp$ has totally geodesic 
leaves (Remark~\ref{fibme}), while the curves $\,s\mapsto x(t,s)\,$ are 
geodesics, tangent to $\,\dzp$ at $\,s=0$. Finally, $\,x_{st}$ is tangent to 
$\,\dzp\nnh$, since so is $\,x_s$ and $\,\dzp$ is parallel.

Furthermore, for the covariant derivatives $\,R_t,R_s$ of the curvature tensor,
\begin{equation}\label{fso}
\mathrm{a)}\hskip14ptR_t(x_t,x_s)x_s\,=\,\fh_t\hs g(x_s,x_s)\hs u\hs, 
\hskip22pt\mathrm{b)}\hskip14ptR_s\,=\,\hs0\hs, 
\hskip22pt\mathrm{c)}\hskip14pt\fh_s\,=\,\hs0\hs.
\end{equation}
Namely, by (f) in Section~\ref{potb}, $\,\fh\,$ is a function of $\,t$, while 
$\,t_s=0\,$ as the curves $\,s\mapsto x(t,s)$ are tangent to 
$\,\dzp\nh=\hs\text{\rm Ker}\,dt$, and (\ref{fso}.c) follows. Since 
$\,\ri=(2-n)\hs\fh(t)\,dt\otimes dt$, $\,\nabla W=0$ and $\,\nabla dt=0$, 
(\ref{fso}.a) and (\ref{fso}.b) are immediate from (\ref{fso}.c) and 
Lemma~\ref{rwcod}(a).

We can now prove (\ref{xtt}). Relation (\ref{xtt}.ii) is obvious from 
(\ref{fso}.c), as $\,g,\gm\,$ and $\,A\,$ are parallel, so that (\ref{uxo}.b) 
yields $\,[\gm(A\uxs,\hs\uxs)]_s=[g(x_s,x_s)]_s=0$. Denoting by $\,\widetilde v\,$ the left-hand side of 
(\ref{xtt}.i), we get $\,\widetilde v=\widetilde v_s=0\,$ at $\,s=0\,$ (from 
(\ref{xts}.a), (\ref{ode}) and (\ref{xtt}.ii) with $\,u_s=v_s=0$). Finally, 
$\,\widetilde v_{ss}=0\,$ for all $\,(t,s)$, since (\ref{xts}.b) and 
(\ref{ruv}) give $\,x_{ttss}=Q(x_s)\hh u/2$. More precisely, according to 
(\ref{xts}.b) with $\,x_{ss}=0\,$ (see (\ref{uxo}.b)), $\,x_{ttss}$ equals the 
sum of three curvature terms, so that, using the Leibniz rule with 
$\,x_{ts}=x_{st}$, we obtain 
$\,x_{ttss}=3\hh R(x_t,x_s)x_{st}+R_t(x_t,x_s)x_s$, all the other terms being 
zero as a consequence of Lemma~\ref{csnpd}(iii) combined with 
(\ref{uxo}.b,\hs c), and (\ref{fso}.b). Using (\ref{ruv}) with 
(\ref{uxo}.a,\hs c) and (\ref{fso}.a), we now get $\,x_{ttss}=Q(x_s)\hh u/2\,$ 
and $\,\widetilde v_{ss}=0$.

Thus, $\,\widetilde v_s$ must be identically zero, as it is parallel in the 
$\,s\,$ direction and vanishes at $\,s=0$. For the same reason, 
$\,\widetilde v=0\,$ for all $\,(t,s)$, which yields (\ref{xtt}.i).

By (\ref{xtt}.i), $\,x_{tt}=0\,$ when $\,s=1$, so that the curve 
$\,t\mapsto x(t,1)\,$ is a geodesic defined on $\,\bbR\hs$. Such geodesics 
realize all initial conditions $\,(x,\dot x)\,$ with the velocities 
$\,\dot x\,$ for which $\,g_x(u_x,\dot x)=1\,$ (the normalization being due to 
the fact that they are parametrized by the function $\,t$, cf.\ 
Remark~\ref{param}). Namely, we can realize 
$\,(x,\dot x)\,$ by the curve $\,t\mapsto y(t)\,$ chosen above, and then use 
the solution $\,w\,$ to (\ref{ode}) with the zero initial conditions. 
\end{proof}

\section{Proof of Theorem~\ref{class}}\label{pottc}
Suppose that $\,(\hm,\hg\hs)\,$ satisfies the assumptions of 
Theorem~\ref{class}. The data required for the construction in 
Section~\ref{xpls} can be introduced as follows. For $\,u\,$ chosen at the 
beginning of Section~\ref{clth}, we define $\,\fh\,$ as in the lines following 
(\ref{ruv}). According to (f) in Section~\ref{potb}, $\,\fh\,$ is, locally, a 
function of $\,t$. The word `locally' may be dropped as we are assuming 
condition (c) in Lemma~\ref{clext}, and hence the level sets of $\,\hs t\hs$ 
are connected. Next, we let $\,\mv\hs$ be the space of all parallel sections 
of $\,\qt$, so that $\,\dim\mv\nh=n-2\,$ by Lemma~\ref{csnpd}(iv). 
Finally, the pseu\-\hbox{do\hs-}Euclid\-e\-an inner product $\,\lr\,$ in 
$\,\mv$ and $\,A:\mv\to\mv\hs$ are the objects induced, in an 
obvious manner, by the fibre metric $\,\gm\,$ on $\,\qt\,$ and the bundle 
morphism $\,A:\qt\to\qt$, both of which are parallel (see Remark~\ref{fibme} 
and (d) in Section~\ref{potb}).

We now fix a null geodesic $\,\bbR\ni t\mapsto x(t)\in\hm\,$ parametrized by 
the function $\,t$, which exists in view of Lemma~\ref{compl} and 
Remark~\ref{param}(a). As $\,g(\dot x,u)=1\,$ (see Remark~\ref{param}(c)), the 
plane $\,P\,$ in $\,\txohm\,$ spanned by the null vectors $\,\dot x(0)\,$ and 
$\,u_{x(0)}$ is $\,g_x$-non\-de\-gen\-er\-ate, with the sign pattern 
\hbox{$\,-\hskip1.2pt+\hs$}, so that $\,\txohm=P\oplus\widetilde\mv\nh$, for 
$\,\widetilde\mv\nh=P^\perp\nnh$.

We define a mapping $\,F:\rto\nh\times\mv\nh\to\hm\,$ by 
$\,F(t,s,v)=\exp_{\hs x(t)}(\tilde v(t)+su_{x(t)}/2)$, for the parallel field 
$\,t\mapsto \tilde v(t)\in\txthm\,$ with 
$\,\tilde v(0)=\hs\text{\rm pr}\hskip2.4ptv_{x(0)}$, where 
$\,\hs\text{\rm pr}:\txohm\to\widetilde\mv\hs$ is the orthogonal projection 
(and so $\,\hs\text{\rm pr}\hh$, restricted to $\,\dzxp$ for $\,x=x(0)$, 
descends to the quotient $\,\qt_x=\dzxp\nnh/\dz_x$, forming an isomorphism 
$\,\qt_x\to\widetilde\mv$).

That $\,F\,$ is a \feo\ can be seen as follows. The manifold $\,N\hs$ in 
Lemma~\ref{clext}(c) is simply connected, since so is $\,\hm$. Therefore,  
each leaf of $\,\dzp$ (level set of $\,t$), with its complete flat 
tor\-sion\-free connection (Remark~\ref{fibme}), is the \feic\ image 
of its tangent space at any point under the exponential mapping, 
cf.\ Remark~\ref{scfco}.

Finally, according to \cite[Lemma 5.1]{derdzinski-roter-07}, $\,F^*\hg\,$ 
coincides with the metric $\,\kx\,dt^2\nh+\hs dt\hs ds+\vh$ on 
$\,\rto\nh\times\mv\nh$, constructed from the data described above as in 
Section~\ref{xpls}.

\section{Proof of Theorem~\ref{nolor}}\label{potc}
In the following lemma, the bundle morphism $\,A:\qt\to\qt\,$ defined in 
Section~\ref{potb} makes sense even without assuming that the line bundle 
$\,\dz\,$ is trivial. In fact, $\,A\,$ depends quadratically on our fixed 
nonzero parallel section $\,u\,$ of $\,\dz$, which, for nontrivial $\,\dz$, 
is still well-defined, locally, up to a sign. (Cf.\ Theorem~\ref{parvf}.)
\begin{lemma}\label{eglfl}Let\/ $\,(M,g)\,$ be a compact \ecs\ 
Lo\-rentz\-i\-an \mf\ of dimension $\,n\ge4\,$ such that, at some/every 
point\/ $\,x\in M$, the nonzero traceless self-ad\-joint endomorphism 
$\,A_x:\qt_x\to\qt_x$ has $\,n-2=\dim\qt_x$ distinct eigenvalues. Then the 
leaves of the parallel distribution $\,\dzp$ are all complete.
\end{lemma}
In fact, in view of Lemma~\ref{csnpd}(iv) and (e) in Section~\ref{potb}, 
passing to a suitable finite covering of $\,M\,$ we may assume that both 
vector bundles $\,\dz\,$ and $\,\qt\hs$ over $\,M\,$ are trivialized by their 
parallel sections. Our assertion now follows from Lemma~\ref{flcpl} applied to 
$\,\lz\,$ which is the restriction of $\,\dz\,$ to 
any leaf $\,N\hs$ of $\,\dzp\nnh$.
\begin{remark}\label{eigvl}For $\,n=4$, the assumption about eigenvalues in 
Lemma~\ref{eglfl} is redundant: it follows from the other stated properties of 
$\,A_x$, since $\,\gm_x$ is positive definite (Remark~\ref{fibme}).
\end{remark}
\begin{proof}[Proof of Theorem~{\rm\ref{nolor}}]Suppose that $\,(M,g)\,$ 
is a compact \ecs\ Lo\-rentz\-i\-an four\mfd. Its universal covering 
$\,(\hm,\hg\hs)\,$ then satisfies the assumptions of Theorem~\ref{class}: the 
leaves of $\,\dzp$ are complete by Lemma~\ref{eglfl} (cf.\ Remark~\ref{eigvl}), 
while condition (c) in Lemma~\ref{clext} holds for $\,t\,$ in view of 
Remark~\ref{concl}.

The conclusion of Theorem~\ref{class} now contradicts 
Lemma~\ref{ecsrr}(iv-b).
\end{proof}

\section{Vector bundles related to Killing fields}\label{vbkf}
Throughout this section, $\,(M,g)\,$ stands for a fixed \ecs\ \prd\ of 
dimension $\,n\ge4\,$ such that the \od\ $\hs\dz\hs$ is two\diml, and 
$\,\hs\om\hs\,$ is the $\,2$-form described in Remark~\ref{oldis}. We assume 
that $\,\hs\om\hs\,$ is single-val\-ued, rather than being defined just up to 
a sign, which can always be achieved by passing to a two-fold covering 
manifold, if necessary.

In addition to $\,\dz\,$ and $\,\dzp\nnh$, we consider here the real vector 
bundle $\,\yv\,$ over $\,M$, the sections of which are the differential 
$\,2$-forms $\,\zh\,$ such that $\,\zh(v,\,\cdot\,)=0\,$ for every section 
$\,v\,$ of $\,\dz$. Thus, $\,\yv\,$ is a subbundle of 
$\,(\tam)^{\wedge2}\nnh$, isomorphic to $\,([\hh\tm/\dz]^*)^{\wedge2}\nnh$, 
and $\,\hs\om\hs\,$ is a section of $\,\yv\,$ (as 
$\,\dz\subset\hs\text{\rm Ker}\,\hs\om\hs\,$ by (\ref{krp}.a) and (\ref{owo}). 
Let $\,\lz\,$ be the real-line subbundle of $\,\yv$ spanned by 
$\,\hs\om\hh$. The \lcc\ of $\,g\,$ induces connections, all denoted by 
$\,\nabla$, in each of the byndles $\,\dz,\dzp\nnh,\yv\,$ and $\,\lz$.

The formula 
$\,\nao_{\!u}(v,\zh)=(\nabla_{\!u}v-\zh u,\nabla_{\!u}\zh
-(n-2)^{-1}g(v,\,\cdot\,)\wedge\ri\hh(u,\,\cdot\,))$, with $\,\ri\,$ standing 
for the \rt\ of $\,g$, clearly defines a connection $\,\nao\,$ in the \vb\ 
$\,\dzp\nnh\oplus\yv$. (Here $\,\zh u\,$ is the unique vector field with 
$\,g(\zh u,\,\cdot\,)=\zh(u,\,\cdot\,)$, and similarly for the symbols 
$\,\ri u,\hs\om_x$ appearing below.) The following facts will be needed in 
the next section.
\begin{enumerate}
  \def\theenumi{{\rm\alph{enumi}}}
\item The connection $\,\nao\,$ in $\,\dzp\nnh\oplus\yv\,$ is flat.
\item The sub\bd\ $\,\dz\oplus\lz\,$ of $\,\dzp\nnh\oplus\yv\,$ is 
$\,\nao$\prl.
\item For every $\,\nao$\prl\ section $\,(v,\zh)\,$ of $\,\dzp\nnh\oplus\yv$, 
the $\,g$-co\-var\-i\-ant derivative $\,\nabla v\,$ is $\,\nabla$\prl\ along 
$\,\dzp\nnh$.
\end{enumerate}
If $\,M\,$ is also simply connected, then, for the space $\,\vy\,$ of all 
$\,\nao$\prl\ sections of $\,\dz\oplus\lz$,
\begin{enumerate}
  \def\theenumi{{\rm\alph{enumi}}}
\item[(d)] $\dim\vy=3$,
\item[(e)] there exists a unique $\,C^\infty$ mapping 
$\,F:M\to\vy\smallsetminus\nh\{0\}\,$ such that $\,F(x)$, for $\,x\in M$, is 
characterized by $\,F(x)=(v,\zh)\,$ with $\,v_x=0\,$ and $\,\zh_x=\hs\om_x$,
\item[(f)] for $\,F\,$ as in (e) and any $\,x\in M$, the differential 
$\,dF_x:\txm\to\vy\hs$ has rank $\,2$, while the image $\,dF_x(\txm)\,$ and 
$\,F(x)\,$ together span $\,\vy$,
\item[(g)] the leaves of $\,\dzp$ are the connected components of nonempty 
$\,F$-pre\-im\-ages of points of $\,\vy$,
\item[(h)] for every leaf $\,N\hs$ of $\,\dzp\nnh$, the tangent bundle 
$\,\tn\hs$ is trivialized by its $\,\nabla$\prl\ sections (cf.\ 
Remark~\ref{fibme}).
\end{enumerate}
We will not use the easily-verified fact that the assignment 
$\,v\mapsto(v,\zh)$, with $\,\zh\,$ characterized by 
$\,\zh(u,\,\cdot\,)=g(\nabla_{\!u}v,\,\cdot\,)\,$ for all vector fields 
$\,u$, is a linear isomorphism between the space of all \kf s $\,v\,$ on 
$\,(M,g)\,$ tangent to $\,\dzp\nnh$, and the space of all $\,\nao$\prl\ 
sections $\,(v,\zh)\,$ of $\,\yz$.

Let $\,\ro,R\,$ and $\,\hat R\,$ be the curvature tensors of $\,\nao$, the 
\lcc\ $\,\nabla\,$ of $\,g\,$ and, respectively, the connection induced by 
$\,\nabla\,$ in $\,[\tam]^{\wedge2}\nnh$. For arbitrary vector fields 
$\,u,u\hh'$ on $\,M$, (\ref{cur}) yields 
$\,\ro(u,u\hh'\hs)\hh(v,\zh)=(v\hh'\nh,\zh\hs'\hs)$, where 
$\,v\hh'\nh=R(u,u\hh'\hs)\hh v
-\text{\hbox{$(n-2)^{-1}g(v,u)\ri u\hh'\nh$}}
+\text{\hbox{$(n-2)^{-1}g(v,u\hh'\hs)\ri u$}}\,$ and 
$\,\zh\hs'\nh=\hat R(u,u\hh'\hs)\hh\zh
+(n-2)^{-1}[\hs\zh(u,\,\cdot\,)\wedge\ri(u\hh'\nnh,\,\cdot\,)
-\zh(u\hh'\nnh,\,\cdot\,)\wedge\ri(u,\,\cdot\,)\hh]$. (By 
Lemma~\ref{rwcod}(b), $\,(\nabla_{\!u}\ri)(u\hh'\nnh,\,\cdot\,)\,$ is 
symmetric in $\,u,u\hh'\nnh$.) Now Lemma~\ref{rwcod}(a) gives 
$\,g(v\hh'\nnh,\,\cdot\,)=W(u,u\hh'\nnh,v,\,\cdot\,)$. (Note that 
$\,\ri(v,\,\cdot\,)=0\,$ in view of (\ref{krp}.a), since $\,v\,$ is a section 
of $\,\dzp\nnh$.) Consequently, $\,v\hh'\nh=0$, as $\,v\,$ is a section of 
$\,\dzp\nh=\hs\text{\rm Ker}\,\hs\om\hs\subset\hs\kw\nh$, by (\ref{owo}) with 
$\,W\nh=\hs\ve\hskip1pt\om\nh\otimes\om\hh$.
Furthermore, by the Ric\-ci identity, 
$\,[\hh\hat R(u,u\hh'\hs)\hh\zh\hs](w,w\hh'\hs)
=\zh(R(u\hh'\nnh,u)\hh w,w\hh'\hs)+\zh(w,R(u\hh'\nnh,u)\hh w\hh'\hs)$ 
for any vector fields $\,w,w\hh'\nnh$. Replacing $\,R\,$ here by the 
expression in Lemma~\ref{rwcod}(a), we get $\,\zh\hs'\nh=0$. (In fact, 
numerous terms vanish since the images of $\,W_{\nh x}$ and $\,\ri_x$ at 
any point $\,x\,$ are contained in 
$\,\dzx\subset\hs\text{\rm Ker}\,\zh_x$, cf.\ (\ref{owo}) and 
(\ref{krp}.a).) This proves (a).

Next, if $\,(v,\zh)\,$ is a section of $\,\dz\oplus\lz$, then, for any vector 
field $\,u$, so are $\,(\nabla_{\!u}v,\nabla_{\!u}\zh)$ (since $\,\dz\,$ and 
$\,\hs\om\hs\,$ are parallel), and 
$\,(\zh u,(n-2)^{-1}g(v,\,\cdot\,)\wedge\ri\hh(u,\,\cdot\,)\,$ (as $\,\dz\,$ 
is the image of $\,\hs\om\hs\,$ by (\ref{owo}), and $\,\zh\,$ equals a 
function times $\,\hs\om\hh$, while $\,v\,$ and $\,\ri u\,$ are sections of 
$\,\dz$, cf.\ Lemma~\ref{csnpd}(ii)). Consequently, $\,\nao_{\!u}(v,\zh)\,$ is 
a section of $\,\dz\oplus\lz\,$ as well, and (b) follows.

By the definition of $\,\nao$, if $\,(v,\zh)\,$ is a $\,\nao$\prl\ section of 
$\,\dzp\nnh\oplus\yv$, then $\,\nabla_{\!u}\zh=0\,$ for every section $\,u\,$ 
of $\,\dzp\nnh$, as one then has $\,\ri\hh(u,\,\cdot\,)=0\,$ by (\ref{krp}.a). 
Hence $\,\zh\,$ is $\,\nabla$\prl\ along $\,\dzp\nnh$, which yields (c).

Now let $\,M\,$ be simply connected. Assertion (d) is obvious from (a) and 
(b), as $\,3\,$ is the fibre dimension of $\,\dz\oplus\lz$. Similarly, (e) 
is due to the existence of a unique parallel section of $\,\dz\oplus\lz\,$ 
with a prescribed value at $\,x$. 

To prove (f), we first observe that $\,dF_x$ sends any $\,u\in\txm\,$ to 
$\,(\dot v,\dot\zh)\in\vy\,$ characterized by 
$\,\dot v_x=-\hs\om_xu\,$ and $\,\dot\zh_x=0$. In fact, let 
$\,t\mapsto x(t)\,$ be a curve in $\,M$, and let us set 
$\,(v(t),\zh(t))=F(x(t))$. Suppressing the dependence on $\,t$, and 
differentiating, covariantly along the curve, both the relation $\,v_x\nh=0\,$ 
and the equality which states that $\,(\nabla v)_x$ corresponds via $\,g_x$ to 
$\,\hs\om_x$, we get $\,\dot v_x+\hs\om_x(\dot x,\,\cdot\,)=0\,$ and 
$\,(\nabla\dot v)_x=0$, as required. (The second covariant derivative of the 
\kf\ $\,v\,$ at $\,x\,$ depends linearly on $\,v_x$, due to a well-known 
identity \cite[formula (17.4) on p.\ 536]{dillen-verstraelen}, and so 
$\,\nabla(\nabla v)=0\,$ at $\,x$, since $\,v_x=0$.) As 
$\,\hs\text{\rm rank}\,\hs\om\hs=2\,$ (see Remark~\ref{oldis}), this 
implies (f).

As a consequence of (c), $\,F\,$ is constant along $\,\dzp\nnh$. (Note that 
$\,\nabla\hs\om\hs=0\,$ and, by (\ref{owo}), 
$\,\dzp\nh=\hs\text{\rm Ker}\,\hs\om\hh$.) Now (g) is immediate from (f).

Finally, in view of (a), given a leaf $\,N\hs$ of $\,\dzp\nnh$, a point 
$\,x\in N\nh$, and a vector $\,w\in\txn=\dzxp\nnh$, we may choose a 
$\,\nao$\prl\ section $\,(v,\zh)\,$ of $\,\dzp\nnh\oplus\yv\,$ satisfying the 
initial conditions $\,v_x=w\,$ and $\,\zh_x=0\,$ (that is, 
$\,(\nabla v)_x=0$). Thus, $\,v\,$ is tangent to $\,N\hs$ and, by (c), 
$\,\nabla$\prl\ along $\,N\nh$, which proves (h).

\section{Proof of the second part of Theorem~\ref{pontr}}\label{psta}
We need the following two simple facts from topology.
\begin{lemma}\label{abfgp}If the fundamental group\/ $\,\Gm$ of a 
compact\/ $\,k$\diml\ \mf\/ $\,P$ is Abelian and the universal 
covering \mf\ of\/ $\,P\nh\,$ is diffeomorphic to $\,\rk\nh$, then\/ 
$\,\Gm$ is isomorphic to $\,\bbZ^k\nh$.
\end{lemma}
\begin{proof}As $\,\Gm\hh$ is tor\-sion\-free by Smith's theorem 
\cite[p.\ 287]{hu}, and finitely generated, it is isomorphic to 
$\,\bbZ^{\hs r}$ for some integer $\,r\ge1$. The $\,K(\bbZ^{\hs r}\nh,1)\,$ 
space $\,P\,$ must have the homotopy type of the $\,r$-to\-rus 
\cite[pp.\ 93--95]{switzer}, so that $\,r=k$, since both $\,r\,$ and $\,k\,$ 
are equal to the highest integer $\,m\,$ with $\,H_m(P,\bbZ_2)\ne\{0\}$.
\end{proof}
\begin{lemma}\label{infgp}If\/ $\,M\to S^2$ is a fibration and its fibre 
$\,N\hs$ is a compact manifold of dimension $\,k\ge2\,$ with a universal 
covering space \feic\ to $\,\rk\nnh$, then the fundamental group of\/ 
$\,M\,$ is infinite.
\end{lemma}
\begin{proof}Since $\,\bbZ=\pi_2S^2\to\hs\pi_1N\to\pi_1M\,$ is a part of 
the homotopy exact sequence of the fibration $\,M\to S^2\nnh$, if 
$\,\pi_1M\,$ were finite, the image $\,\Gm\hs$ of $\,\bbZ=\pi_2S^2$ would 
be a cyclic subgroup of finite index in $\,\pi_1N\nh$. The manifold 
$\,P=\rk\nh/\Gm$, forming a finite covering space of 
$\,N\nh=\rk/\pi_1N\nh$, would be compact, which, as $\,k\ge2$, would 
contradict Lemma~\ref{abfgp}.
\end{proof}
We now assume that $\,(M,g)\,$ is a compact simply connected \ecs\ \mf. As we 
show below, this assumption leads to a contradiction, which proves the claim 
about $\,\pi_1M\,$ in Theorem~\ref{pontr}.

Let $\,\rd\in\{1,2\}\,$ be the dimension of the \od\ $\,\dz\,$ (see 
Lemma~\ref{csnpd}(i)).

If $\,\rd=1$, Lemma~\ref{csnpd}(iv) implies the existence of a nonzero 
global parallel vector field $\,u\,$ spanning $\,\dz$. The $\,1$-form 
$\,\xi=g(u,\,\cdot\,)$, being parallel, is closed, so that $\,\xi=dt\,$ 
for some function $\,t$. As $\,dt\,$ is parallel, $\,dt\ne0\,$ 
everywhere, which contradicts compactness of $\,M$.

Now let $\,\rd=2$, and let $\,\vy,F\,$ be as in (d) -- (g), 
Section~\ref{vbkf}. The formula $\,\pi(x)=F(x)/|F(x)|\,$ defines a submersion 
$\,\pi:M\to S^2\nnh$, where $\,S^2\nh=\{w\in\vy:|w|=1\}\,$ is the unit sphere 
for a fixed Euclidean norm $\,|\hskip2.4pt|\,$ in the $\,3$-space $\,\vy$. As 
$\,\pi(M)\subset S^2$ is both compact and open, the submersion $\,\pi\,$ is 
surjective, so that $\,\pi\,$ is a fibration (Remark~\ref{fibra}). The fibres of 
$\,\pi\,$ thus are the leaves of $\,\dzp$ (see (g) in Section~\ref{vbkf}). If 
a fixed fibre $\,N\hs$ is endowed with the flat tor\-sion\-free connection 
mentioned in Remark~\ref{fibme}, then, according to (h) in Section~\ref{vbkf}, 
$\,\tn\hs$ is trivialized by its parallel sections. Lemma~\ref{flcpl} (for 
$\,\lz=\tn$) and Remark~\ref{scfco} now imply that the universal covering 
manifold of $\,N\hs$ is \feic\ to $\,\bbR^{\hskip-.6ptn-2}\nnh$. Since 
$\,\pi_1M\,$ was assumed to be trivial, and $\,n-2\ge2$, this contradicts 
Lemma~\ref{infgp}, thus completing the proof of Theorem~\ref{pontr}.

\section{Further remarks}\label{fure}
This section consists of three separate comments, indicating how some results 
presented above might be strengthened.

First, Theorem~\ref{class}, with essentially the same proof, remains valid if, 
in its assumptions, condition (c) in Lemma~\ref{clext} and completeness of the 
leaves of $\,\dzp$ are replaced by completeness of $\,\hg$.

Secondly, the argument that we used to show nonexistence of compact 
four\diml\ \ecs\ Lo\-rentz\-i\-an \mf s can be minimally modified so as to 
yield the following classification theorem: {\it If a compact \ecs\ 
Lo\-rentz\-i\-an \mf\/ $\,(M,g)\,$ of any dimension $\,n\ge4\,$ satisfies the 
assumption about distinct eigenvalues made in Lemma\/ {\rm\ref{eglfl}}, then 
the \psr\ universal covering $\,(\hm,\hg\hs)\,$ of\/ $\,(M,g)\,$ coincides, up 
to an isometry, with one of the manifolds constructed in 
Section\/~{\rm\ref{xpls}}, and the fundamental group of\/ $\,M$, treated as a 
group of isometries of $\,(\hm,\hg\hs)$, has a fi\-nite-index subgroup 
contained in the group $\,\gp\hs$ defined in Section\/~{\rm\ref{xpls}}}.

Finally, in Section~\ref{psta} we showed that $\,\pi_1M\,$ is infinite for any 
compact \ecs\ \mf\ $\,(M,g)$, using separate arguments for the cases 
$\,\rd=1\,$ and $\,\rd=2$, where $\,\rd\,$ is the dimension of the \od\ 
$\,\dz$. If $\,\rd=1$, we obtain the stronger conclusion $\,b_1(M)\ge1\,$ from 
the following lemma applied to $\,\mathcal{F}=\dzp$ along with a 
vec\-tor-bun\-dle isomorphism $\,\dz\to(\tm/\mathcal{F})^*$ provided by $\,g$, 
cf.\ Lemma~\ref{csnpd}(iv).
\begin{lemma}\label{betti}Let a compact \mf\/ $\,M\hs$ with a 
tor\-sion\-free connection $\,\nabla\hs$ admit a co\-di\-men\-sion-one 
$\,\nabla$\prl\ distribution\/ $\,\mathcal{F}$ such that the 
connection induced by $\,\nabla\hs$ in the quotient line bundle\/ 
$\,\hs\tm/\mathcal{F}\,$ is flat. Then the first Betti number\/ $\,b_1(M)\hs$ 
is positive, and, consequently, $\,M\hs$ has an infinite fundamental group.
\end{lemma}
\begin{proof}The dual bundle of $\,\hs\tm/\mathcal{F}\hs$ may be 
identified with the real-line subbundle $\,\lz$ of $\,\tam\,$ 
such that the sections of $\,\lz\,$ are precisely those $\,1$-forms 
$\,\xi\,$ on $\,M\,$ with $\,\xi(w)=0$ for every section $\,w\,$ of 
$\,\mathcal{F}$. Clearly, $\,\lz\,$ is $\,\nabla$\prl.

Suppose, on the contrary, that $\,b_1(M)=0$, so that the homology 
group $\,H_1(M,\bbZ)\,$ is finite. Replacing $\,M\,$ by a two-fold covering 
\mf, if necessary, we may assume that $\,\lz\,$ is spanned by a global nonzero 
parallel $\,1$-form $\,\xi$. In fact, the connection in $\,\lz\,$ induced by 
$\,\nabla\,$ is flat, and so its holonomy representation, with any fixed 
base point $\,x\in M$, its valued in the multiplicative group 
$\,\bbR\smallsetminus\{0\}$. Since $\,\bbR\smallsetminus\{0\}\,$ is Abelian, 
the holonomy representation is a composite 
$\,\pi_1M\to H_1(M,\bbZ)\to\bbR\smallsetminus\{0\}$, and its image must, due 
to finiteness of $\,H_1(M,\bbZ)$, be contained in $\,\{1,-1\}$.

As $\,\nabla\,$ is tor\-sion\-free and $\,b_1(M)=0$, the parallel 
$\,1$-form $\,\xi\,$ is closed, and hence $\,\xi=dt$ is nonzero 
everywhere, for some function $\,t$, which contradicts compactness of $\,M$.
\end{proof}

\end{document}